\documentclass[11pt]{amsart}

\usepackage{amsmath,amsthm,amssymb,latexsym,epic,bbm,comment,yfonts,mathrsfs}
\usepackage{graphicx,enumerate,stmaryrd,rotating,xcolor,ytableau}
\usepackage[all]{xy}
\xyoption{2cell}

\allowdisplaybreaks

\usepackage[active]{srcltx}

\usepackage{ytableau}

\usepackage[latin1]{inputenc}
\usepackage{amsxtra}
\usepackage{amsmath}
\usepackage{amscd}
\usepackage{amssymb}
\usepackage{amsfonts}
\usepackage[all]{xy}
\usepackage{mathrsfs}
\usepackage{amsthm}
\usepackage{color}
\usepackage{upgreek}
\usepackage{verbatim}  
\usepackage{enumerate}
\usepackage{latexsym}
\usepackage{graphicx}
\usepackage{tikz}
\usepackage{todonotes}
\usepackage{setspace}
\usepackage{hyperref}
\usepackage{stmaryrd}
\usepackage{nicefrac}
\usepackage{euscript}
\usetikzlibrary{decorations.pathmorphing}

 \definecolor{darkgreen}{HTML}{336633}
 \definecolor{darkred}{HTML}{993333}
\makeatletter

\newcommand{\arxiv}[1]{\href{http://arxiv.org/abs/#1}{\tt
    arXiv:\nolinkurl{#1}}}

\theoremstyle{plain}
\newtheorem{thm}{Theorem}[section]
\newtheorem*{thm*}{Theorem}

\newtheorem{lem}[thm]{Lemma}
\newtheorem{prop}[thm]{Proposition}

\newtheorem{cor}[thm]{Corollary}
\newtheorem{df-prop}[thm]{Definition-Proposition}

\newtheorem{definition}[thm]{Definition}

\theoremstyle{remark}
\newtheorem{rem}[thm]{Remark}
\newtheorem{ex}[thm]{Example}

\numberwithin{equation}{section}

\def\calW{\mathcal{W}}

\def\Lam{\Lambda}

\def\ufk{\underline{\mathfrak{k}}}
\def\ufh{\underline{\mathfrak{h}}}

\def\mod{\operatorname{-mod}\nolimits}

\def\ep{\epsilon}
\def\gl{\mathfrak{gl}}
\def\osp{\mathfrak{osp}}

\def\la{\lambda}

\def\ov{\overline}

\newcommand{\mc}{\mathcal}
\newcommand{\mf}{\mathfrak}
\newcommand{\C}{\mathbb C}

\newcommand{\oa}{{\bar 0}}
\newcommand{\ob}{{\bar 1}}

\newcommand{\ad}{\mathrm{ad}}

\newcommand{\fk}{\mathfrak{k}}
\newcommand{\fg}{\mathfrak{g}}

\newcommand{\fb}{\mathfrak{b}}

\newcommand{\n}{\mathfrak{n}}

\newcommand{\cO}{\mathcal{O}}

\newcommand{\h}{\mathfrak{h}}

\newcommand{\g}{\mathfrak{g}}
\newcommand{\sg}{\ov{\mathfrak{g}}}

\newcommand{\fl}{\mathfrak{l}}

\newcommand{\fu}{\mathfrak{u}}

\newcommand{\Z}{{\mathbb Z}}

  \newcommand{\W}{\mf{Wh}}
 \newcommand{\MS}{\mc{MS}}

\newcommand{\VV}{\mathbb V}
\newcommand{\WW}{\mathbb W}

\begin{document}

\title{Super duality for Whittaker modules and finite $W$-algebras}

\author[Cheng]{Shun-Jen Cheng}
\address{Institute of Mathematics, Academia Sinica, Taipei, Taiwan 10617} \email{chengsj@math.sinica.edu.tw}
\author[Wang]{Weiqiang Wang}
\address{Department of Mathematics, University of Virginia, Charlottesville, VA 22903, USA}\email{ww9c@virginia.edu}


\begin{abstract}
We establish a super duality as an equivalence between Whittaker module categories over a pair of classical Lie algebra and Lie superalgebra in the infinite-rank limit. Building on this result and utilizing the Losev-Shu-Xiao decomposition, we obtain a super duality which is an equivalence between module categories over a pair of finite $W$-algebras and $W$-superalgebras at the infinite-rank limit.
\end{abstract}

\maketitle

\setcounter{tocdepth}{1}
\tableofcontents

%
\section{Introduction}

\subsection{}
Super duality refers to an equivalence of categories between parabolic BGG categories of modules over a suitable pair consisting of a reductive Lie algebra and a basic classical Lie superalgebra in the infinite-rank limit (cf. \cite[Chapter 6]{CW12}). This concept was first formulated as a conjecture in type $A$ \cite{CW08}, generalizing the maximal parabolic case in \cite{CWZ08}, and was established in \cite{CL10}; the duality has since been extended to other classical types in \cite{CLW11} and then to Kac-Moody setting \cite{CKW15}. Super duality provides character formulas for irreducible or tilting modules in the suitable parabolic BGG categories of Lie superalgebras. Moreover, it has led to a proof of Brundan-Kazhdan-Lusztig conjecture for the general linear Lie superalgebras \cite{Br03, CLW15} (also cf. \cite{BLW17}) and the development of a super Kazhdan-Lusztig theory for ortho-symplectic Lie superalgebras \cite{BW18KL, Bao17}.

\subsection{}

The goal of this paper is to extend the concept of super duality to the settings of Whittaker modules as well as finite $W$-superalgebras.

A finite $W$-(super)algebra $U(\g,e)$ is an associative (super)algebra constructed out of a reductive Lie (super)algebra $\g$ and an even nilpotent element $e\in \g$. They can be viewed as generalizations of the universal enveloping algebras of Lie superalgebras as $U(\g,0)$ reduces to $U(\g)$.
A generalization of BGG category $\mc O$ for finite $W$-algebras has been formulated in \cite{BGK08}, with further parabolic generalizations in \cite{BG13, Los12}.

On the other hand, Whittaker $\g$-modules are associated with a nilpotent character $\zeta$ in a reductive Lie algebra $\g$ \cite{Kos78, McD85}. Suitable categories of Whittaker $\g$-modules, denoted $\MS(\zeta)$ or MMS categories \cite{McD85, MS97}, generalize the BGG categories when $\zeta=0$. Recent work has extended these constructions to basic Lie superalgebras $\g$ \cite{Ch21, CCM23, CC24}.

\subsection{}

In this paper, we consider the case when $\g$ is either a reductive Lie algebra or a basic classical Lie superalgebra and $\zeta: \mf n \rightarrow \C$ is a character associated with an even Levi subalgebra $\fl$; see \eqref{eq:zeta}.

The Backelin functor $\Gamma_\zeta:\mc O_\Z \rightarrow \MS(\zeta)$ (see \cite{Ba97} and \eqref{eq:BFunctor})  allows us to connect the BGG category $\mc O_\Z$ of integral weight $\g$-modules to the MMS category of $\g$-modules. Denote by ${\mc W}(\zeta) \subset \MS(\zeta)$ the image category of $\Gamma_\zeta$. This category can be identified with a distinguished Serre quotient category $\mc O_\Z/\mc I_\zeta$.
As a superalgebra generalization of the construction in \cite{MaS05}, a properly stratified cokernel subcategory $\mc O^{\zeta\text{-pres}}$ of $\mc O_\Z$ was formulated in \cite{CCM23}, and it was shown that the restriction of the quotient functor $\Gamma_\zeta: \mc O_\Z \rightarrow \mc O_\Z/\mc I_\zeta$ to $\mc O^{\zeta\text{-pres}}$ induces an equivalence $\mc O^{\zeta\text{-pres}} \cong \mc O_\Z/\mc I_\zeta$.

To summarize, we have the following commutative diagram without the superscript $\mf q$ throughout:
\begin{align}  \label{diag:q}
\xymatrix{
 & \mc O^{\mf q}_\Z
 \ar[r]
\ar[d]
\ar[rd]
& \MS(\zeta)^{\mf q}
 \\
\mc O^{\mf q,\zeta\text{-pres}}
\ar[r]^{\cong}
& \mc O^{\mf q}_\Z/\mc I^{\mf q}_\zeta
\ar[r]^{\cong}
& {\mc W}(\zeta)^{\mf q}
\ar@{^{(}->}[u]
}
\end{align}

In this paper, we formulate a parabolic generalization of such a diagram by adding a superscript $\mf q$, resulting in the diagram \eqref{diag:q}. Here $\mf q$ is a parabolic subalgebra of $\g$ associated with an even Levi subalgebra $\fk$ such that the derived subalgebras of $\fk$ and $\fl$ commute. In this way, we can define a parabolic MMS category $\MS(\zeta)^{\mf q}$ of Whittaker $\g$-modules besides the parabolic BGG category $\mc O^{\mf q}$ and its integral weight subcategory $\mc O^{\mf q}_\Z$. We show that Backelin functor $\Gamma_\zeta$ restricts to a functor $\Gamma_\zeta:\mc O^{\mf q}_\Z\longrightarrow \MS(\zeta)^{\mf q}$ with favorable properties (see Proposition~ \ref{prop:gammaq}), and we denote by ${\mc W}(\zeta)^{\mf q}$ the image category of this functor.
(Setting $\mf q$ to be the Borel subalgebra $\mf b$ amounts to dropping the superscript $\mf q$ in \eqref{diag:q}.)

\subsection{}

As in the super duality setting for parabolic BGG categories (cf. \cite{CLW11}, \cite[Chapter~6]{CW08}), we introduce two families of Lie algebras and superalgebras, $\g_n$ and $\sg_n$, for $n\ge 1$. Namely, starting from the same head Dynkin diagram and connecting with 2 tail Dynkin diagrams associated to $\gl(1+n)$ and $\gl(1|n)$ give us the Dynkin diagrams for $\g_n$ and $\sg_n$, respectively. We consider the parabolic BGG category $\mc O_{n,\Z}^+$ of integral weight $\g_n$-modules (and respectively, $\ov{\mc O}_{n,\Z}^+$ of integral weight $\sg_n$-modules) which restrict to polynomial representations over $\gl(n)$. Here, $\gl(n)$ is regarded as subalgebra of the ``tail'' algebras $\gl(1|n)$ and $\gl(1+n)$. We also formulate MMS categories ${\MS}(\zeta)^{+}_n$ and $\ov{\MS}(\zeta)^{+}_n$ of Whittaker modules over $\g_n$ and $\sg_n$ which are polynomial over the same subalgebra $\gl(n)$.

Specializing the diagram \eqref{diag:q} to the setting of Lie algebra $\g_n$ gives us the following commutative diagram (where the superscript $+$ stands for ``polynomial over $\gl(n)$"):
\begin{align}  \label{diag:gn}
\xymatrix{
 & \mc O^{+}_{n,\Z}
 \ar[r]
\ar[d]
\ar[rd]
& \MS(\zeta)^{+}_n
 \\
\mc O^{\zeta\text{-pres}+}_n
\ar[r]^{\cong}
& \mc O^{+}_{n,\Z}/\mc I^{+}_\zeta
\ar[r]^{\cong}
& {\mc W}(\zeta)^{+}_n
\ar@{^{(}->}[u]
}
\end{align}
On the other hand, specializing the diagram \eqref{diag:q} to the setting of Lie superalgebra $\sg_n$ gives us the following commutative diagram:
\begin{align}   \label{diag:sgn}
\xymatrix{
 & \ov{\mc O}^{+}_{n,\Z}
 \ar[r]
\ar[d]
\ar[rd]
& \ov{\MS}(\zeta)^{+}_n
 \\
\ov{\mc O}^{\zeta\text{-pres}+}_n
\ar[r]^{\cong}
& \ov{\mc O}^{+}_{n,\Z}/\ov{\mc I}^{+}_\zeta
\ar[r]^{\cong}
& \ov{\mc W}(\zeta)^{+}_n
\ar@{^{(}->}[u]
}
\end{align}

We can make sense of the above diagrams \eqref{diag:gn}--\eqref{diag:sgn} at $n=\infty$, and we drop the index $n=\infty$ for simplicity.  By super duality \cite{CW12} (and also \cite{CL20, Le20}), we have an equivalence of highest weight categories: ${\mc O}^{+}_\Z \cong \ov{\mc O}^{+}_\Z.$ This equivalence matches the corresponding Serre subcategories ${\mc I}^{+}_\zeta \cong \ov{\mc I}^{+}_\zeta$, yielding an equivalence of properly stratified categories (see Theorem~\ref{thm:SD:MMS})
 \[
    \mc W(\zeta)^+ \stackrel{\cong}{\longrightarrow} \ov{\mc W}(\zeta)^+;
\]
This will be referred to as super duality for Whittaker modules.

\subsection{}
A nilpotent element $e$ and a nilpotent character $\zeta$ can be chosen in a compatible way, up to conjugation, via the Killing form, with $e=0$ corresponding to $\zeta=0$. We choose $e$ and $\zeta$ to be associated with an even Levi subalgebra $\fl$; cf. \eqref{eq:zeta}. Finite $W$-algebra $U(\g,e)$ admits a rich representation theory and has connection to many other areas; see, e.g., \cite{Los10, Wan11} for surveys.

The category of $\g$-modules on which a subalgebra $\mf m_\chi$, associated to $e$ (see \eqref{eq:mchi}), acts nilpotently is equivalent to the module category over finite $W$-algebra $U(\g,e)$ thanks to Skryabin equivalence \cite[Appendix]{Pre02}. Losev developed an approach via Fedosov quantization to establish an equivalence at the level of (variants of) category $\mc O$ of $U(\g,e)$-modules (for reductive Lie algebra $\g$). Losev's approach relies on a fundamental decomposition theorem \cite{Los10JAMS, Los12} on a certain completion $U(\g)_{\widetilde{\mf m}}^{\wedge}$ of $U(\g)$ with respect to $\widetilde{\mf m}$ in \eqref{eq:mchi2}. Shu-Xiao \cite{SX20} developed an algebraic approach for a super generalization of Losev's constructions in \cite{Los10JAMS}, where $\g$ is a basic classical Lie superalgebra (cf. \cite{Zh14, ZS15}).

We refine Losev-Shu-Xiao decomposition in Theorem~\ref{thm:catOW} to suit our purpose for category equivalences below.
Following \cite{Los12}, we then derive equivalences of the corresponding categories of MMS Whittaker $\g$-modules and $U(\g,e)$-modules, for a basic classical Lie superalgebra $\g$ (see Theorem~\ref{thm:isoK} and \eqref{equiv:mcK}). This, together with \cite{Ch21, CCM23, CC24}, implies that the composition multiplicities of Verma modules in the category $\mc O$ of the finite $W$-superalgebra of a basic classical Lie superalgebra are computed by the super Kazhdan-Lusztig polynomials in the usual BGG category (see Remark~ \ref{rem:comp:verma}). When $\g$ is a reductive Lie algebra, such a relationship was conjectured in \cite{BGK08} and proved in \cite{Los10}.

Combining the category equivalences in Theorem~\ref{thm:isoK} with the super duality for Whittaker modules in Theorem~\ref{thm:SD:MMS}, we arrive at a super duality result for a pair of finite $W$-(super)algebras of classical type in the infinite-rank limit (see Theorem~\ref{thm:equiv:walg}).

\subsection{}

This paper is organized as follows.
In Section~\ref{sec:WH}, we formulate the (parabolic) MMS category $\MS(\zeta)^{\mf q}$ of Whittaker $\g$-modules, for a basic classical Lie superalgebra $\g$ and an even nilpotent character $\zeta$. We review basic properties of the Backelin functor and consider its parabolic analogue.

In Section~\ref{sec:para:coker}, we introduce the parabolic cokernel categories $\mc O^{\mf q,\zeta\text{-pres}}$ and complete the commutative diagram \eqref{diag:q}.

Section~\ref{sec:SD:MMS} presents Lie algebras $\g_n$ and Lie superalgebras $\sg_n$ and formulates the diagrams \eqref{diag:gn}--\eqref{diag:sgn}. We describe BGG-type reciprocities and Ringel-type dualities for the properly stratified categories $\mc O^{\zeta\text{-pres}+}$ and $\ov{\mc O}^{\zeta\text{-pres}+}$, and establish super duality for Whittaker modules.

 Section~\ref{sec:Walgebra} contains a review of the construction of finite $W$-superalgebras. In Section~\ref{sec:Wsuper duality}, we formulate Losev-Shu-Xiao decomposition in the superalgebra setting, and establish several equivalences of categories of Whittaker $\g$-modules and $U(\g,e)$-modules. Finally, we derive super duality for finite $W$-algebras.

 \noindent{\bf Acknowledgments.} We thank Chih-Whi Chen and Volodymyr Mazorchuk for their very helpful discussions and explanations on parabolic cokernel categories. We are also grateful to University of Virginia and Academia
Sinica for hospitality and support. SJC is supported by an NSTC grant of the R.O.C. WW is partially supported by DMS--2401351.

\section{Categories of Whittaker modules}
\label{sec:WH}

In this section, we formulate several categories including the MMS category and its parabolic variants of Whittaker modules over a basic classical Lie superalgebra $\g$.

\subsection{Basic setup}
 \label{subsec:setup}

 The setup in this subsection will be used throughout the paper.

Let $\g=\g_\oa \oplus \g_\ob$ be a basic classical Lie superalgebra of classical type, i.e., of type $\gl$ or $\osp$, with a non-degenerate invariant bilinear form  $(\cdot|\cdot)$, e.g., the super trace form. Let $e\in\g$ be an even nilpotent element and $\{e,h,f\}$ be an $\mf{sl}(2)$-triple in $\g$. We let $\h$ be a Cartan subalgebra of $\g$ such that $h\in\h$.

For a subset $\mf a\subseteq\g$, we denote $\g_{\mf a}=\{x\in\g\mid[x,y]=0,\forall y\in\mf a\}$. Let
\begin{align}
    \label{eq:t}
\mf t:=\h_{e}=\{a\in\h\mid [a,e]=0\}.
\end{align}
Recall that a subalgebra $\mf r\subseteq\mf t$ $(=\h_{e})$ is called a {\it full subalgebra} in \cite[\S3.1]{BG13}, if the center of $\mf g_{\mf r}$ is equal to $\mf r$.

Let $T$ be the adjoint group of $\mf t$.
Let $\theta\in\mf t$ be an {\it integral} element, i.e., $\theta$ is an element in the cocharacter of $T$. The element $\theta\in\mf t$ determines a minimal full subalgebra $\mf r$ of $\mf t$, minimal in the sense that $\theta\in\mf r\subseteq\mf t$ and $\theta$ is regular in $\mf r$. We shall assume that
\[
\mf l:=\mf g_{\mf r} \text{ is an {\em even} subalgebra of } \g,
\]
which is then clearly reductive. (This is all we need in the formulation of super duality later on.) The inclusion of subalgebras $0\subseteq\mf r\subseteq\mf t$ gives rise to an inclusion of Lie subalgebras $\g_{\mf t}\subseteq\g_{\mf r}\subseteq \g$. It follows by definition of $\mf t$ in \eqref{eq:t} that $[e,\mf t]=0$, and hence we have
\begin{align}  \label{eq:el}
    e \in \fl.
\end{align}
We have an $\ad\,\theta$-eigenspace decomposition of $\g$:
\begin{align}
\label{eq:eigen}
\begin{split}
    \g &=\bigoplus_{k\in\Z}\g_{\theta,k},
    \\
    \g_{\theta,k} &:=\{x\in\g\mid[\theta,x]=k x\},
    \quad \text{ with } \mf l=\g_{\theta,0}.
\end{split}
\end{align}
Let $\Phi$ be the root system for $(\g,\h)$, and $\g_\alpha$ be the root space for $\alpha \in \Phi$. Note that $\g_\alpha \subseteq \g_{\theta,\alpha(\theta)}.$

We choose a triangular decomposition
\begin{align}
		\g=\mf n^- \oplus \h \oplus \mf n
  \label{eq:tri}
\end{align}
to be compatible with \eqref{eq:eigen} in the following sense: $\alpha (\theta) >0$ for $\alpha\in \Phi$ implies that $\alpha \in \Phi^+$, where $\Pi \subset \Phi^+$ denote the simple system and positive root system for $\g$ associated to $\mf n$ in \eqref{eq:tri}, i.e., $\mf n =\oplus_{\alpha \in \Phi^+} \g_\alpha$. Denote
\begin{align}
\label{eq:Levi}
\Phi_{\fl} =\Phi \cap \{\alpha \in \Phi\mid \alpha(\theta)=0\},
\quad
\Phi_{\fl}^+ =\Phi^+ \cap \Phi_{\fl},
\quad
\Pi_{\fl}=\Pi \cap \Phi_{\fl}.
\end{align}
Then $\fl$ is a Levi subalgebra of $\g$ with simple system $\Pi_{\fl}$, and
\begin{align}
  \label{eq:udecomp}
	\mf g=\mf u^-\oplus \mf l \oplus \mf u,
 \qquad
 \fl= \h\oplus \bigoplus_{\alpha \in \Phi_{\fl}} \fg_\alpha,
\end{align}
where
\begin{align}
 \label{eq:paracomp}
\quad\fu:=\bigoplus_{\alpha(\theta)>0} \fg_\alpha,
\qquad
\fu^-:=\bigoplus_{\alpha(\theta)>0} \fg_{-\alpha}.
\end{align}
Note that $\mf u \subset \mf n$ and $\mf u^- \subset \mf n^-$.
Denote the associated parabolic subalgebra by
\begin{align} \label{eq:para}
    \mf p= \fl +\n =\fl \oplus \fu =\bigoplus_{k\ge 0}\g_{\theta,k}.
\end{align}

Denote the simple root vectors for $\g$ by $E_\alpha$, for $\alpha \in \Pi$. Associated to the root datum \eqref{eq:Levi}, we define a character $\zeta$ of $\n$ by requiring
\begin{align}
  \label{eq:zeta}
  \zeta: \n \longrightarrow \C, \qquad
\zeta(E_\alpha)=
\begin{cases}
 1, & \text{ for } \alpha\in\Pi_\fl,\\
 0, & \text{ for } \alpha\in\Pi\setminus\Pi_\fl.
\end{cases}
\end{align}

\subsection{McDowell-Milicic-Soergel category}
\label{subsec:MMS}


Recall the character $\zeta:\mf n\rightarrow\C$ from \eqref{eq:zeta} associated to $\fl$. The McDowell-Milicic-Soergel (henceforth abbreviated MMS) category (\cite{McD85,MS97}, \cite[\S3.1]{Ch21}), denoted by $\MS(\zeta)$, is the category of finitely generated $U(\g)$-modules on which $x-\zeta(x)$, for $x\in\n$, acts locally nilpotently, and furthermore on which the action of $Z(\g_\oa)$ is locally finite. Here $Z(\g_\oa)$ denotes the center of the universal enveloping algebra $U(\g_\oa)$.

We set $\chi_\la^{\mf l}:Z(\mf l)\rightarrow \C$ to be the central character with kernel equals the annihilator of $Z(\fl)$ on the Verma module of highest weight $\la \in \h^\ast$. Let $\C_\zeta$ denote the one-dimensional $\mf n\cap\mf l$-module (as a restriction of $\zeta$). Also let $W_{\fl}$ denote the Weyl group of $\mf l$.  The standard Whittaker module is defined as follows:
\[
M(\la, \zeta):= U(\g)\otimes_{U(\mf p)} {K(\fl; \la, \zeta)},
\]
where
\[
{K(\fl; \la, \zeta)} :=U(\mf l)/(\text{Ker}\chi_\la^{\mf l}) U(\mf l)\otimes_{U({\mf n}\cap \mf l)}\mathbb C_\zeta
\]
denotes Kostant's simple Whittaker $\fl$-module \cite{Kos78}. One sees from the eigenvalues of $\theta$ that $M(\la,\zeta)$ has a unique maximal submodule and hence a simple head, which we denote by $L(\la,\zeta)$. Furthermore, we have \cite[Theorem 6]{Ch21}
\begin{align}  \label{eq:LMWiso}
{L}(\la, \zeta)\cong  {L}(\mu, \zeta)\Leftrightarrow {M}(\la, \zeta)\cong {M}(\mu, \zeta)\Leftrightarrow W_{\fl} \cdot \la =W_{\fl} \cdot \mu,
\end{align}
for any $\mu \in \h^\ast$, where the dot action of the Weyl group is used. We have $M(\la,\zeta)\in\MS(\zeta)$.

Then the set $\{{L}(\la, \zeta)\mid \la \in \h^* \text{ is }W_\fl\text{-antidominant}\}$ is a complete set of pairwise non-isomorphic simple objects in ${\MS}(\zeta)$.

\subsection{Backelin functor}
\label{sec:backelin:functor}
Let $\g$ be a basic classical Lie superalgebra with triangular decomposition \eqref{eq:tri} and let $\mc O$ be the BGG category of finitely generated $\h$-semisimple $\g$-modules on which the action of $U(\mf n)$ is locally finite. Denote by $M(\g,\la)$ and $L(\g,\la)$, respectively, the Verma and irreducible modules in $\mc O$ of highest weight $\la\in\h^\ast$. When there is no confusion, we write $M(\la)$ and $L(\la)$ for $M(\g,\la)$ and $L(\g,\la)$, respectively.

Let $\zeta:\mf n\rightarrow\C$ be a character of $\n$ as in \eqref{eq:zeta}.

If $M$ is a $\g$-module with weight space decomposition $M=\oplus_{\mu\in\h^*}M_\mu$ where
$\dim M_\mu<\infty$ for each $\mu$, then we let $\underline{M}:=\prod_{\mu\in\h^*}M_\mu$ be its completion. The vector space $\underline{M}$ is naturally a $\g$-module, and hence so is
\begin{align*}
\Gamma_\zeta(M):= \big\{f\in \underline{M}\mid (x-\zeta(x))^kf=0,\forall x\in\mf n, k\gg 0 \big\}.
\end{align*}
Then $M\rightarrow\Gamma_\zeta(M)$ gives rise to an exact functor from the BGG category of $\g$-modules to the category of MMS Whittaker $\g$-modules, called the Backelin functor \cite[Section 3]{Ba97}:
\begin{align}
  \label{eq:BFunctor}
    \Gamma_\zeta:\mc O\longrightarrow \MS(\zeta).
\end{align}

The following proposition in the case of simple Lie algebras was proved in \cite[Proposition 6.9]{Ba97}. For Lie superalgebras it follows from \cite[Theorem~ 20]{Ch21} and \cite[Theorem 6]{CC24}.

\begin{prop}\label{prop:gamma:O}
    The Backelin functor \eqref{eq:BFunctor} satisfies that
    \begin{align*}
        &\Gamma_\zeta\left(M(\la)\right)=M(\la,\zeta),\\
        &\Gamma_\zeta\left(L(\la)\right)=\begin{cases}
            L(\la,\zeta),\text{ if }\la\text{ is }W_{\fl}\text{-antidominant},\\
            0,\qquad\text{ otherwise.}
        \end{cases}
    \end{align*}
\end{prop}

\subsection{Parabolic MMS categories}
\label{subsec:parabMMS}

We shall be interested in certain parabolic subcategories of the category $\MS(\zeta)$ of MMS Whittaker modules, where $\zeta: \n \rightarrow \C$ is the character from \eqref{eq:zeta}.

Suppose that $\mf k$ is another even Levi subalgebra of the Lie superalgebra $\g$ such that the two semisimple subalgebras $[\fl,\fl]$ and $[\fk,\fk]$ commute. Let
\begin{align}  \label{eq:parab q}
\mf q=\mf k+\mf n
\end{align}
denote the parabolic subalgebra corresponding to $\mf k$. We define $\MS(\zeta)^{\mf q}$ to be the full subcategory of $\MS(\zeta)$ consisting of $\g$-modules $M$ such that the action of $U(\mf q)$ is locally finite. In particular, any $M\in\MS(\zeta)^{\mf q}$ is a direct sum of finite-dimensional irreducible $\mf k$-modules.

Let
\[
\mf P=\mf l+\mf k+\mf n
\]
denote the parabolic subalgebra of $\g$ associated with Levi subalgebra $\mf L:=\mf l+\mf k$. Let $\mf z\subseteq\h$ be the center of $\mf L$ so that we have $\h=(\h\cap[\mf L,\mf L])\oplus\mf z$. Let $\la\in\h^\ast$ be such that $\la|_{\h\cap[\fk,\fk]}$ is dominant integral. We have a decomposition $\la=\la|_{\h\cap[\mf L,\mf L]}+\la^\perp$, where $\la^\perp\in\mf z^\ast$ which vanishes on $\h\cap[\mf L,\mf L]$. We form the irreducible $\mf L$-module
\begin{align*}
    K([\fl,\fl]; \la|_{\h\cap[\fl,\fl]},\zeta)\otimes L([\fk,\fk];\la|_{\h\cap[\fk,\fk]})\otimes\C_{\la^\perp},
\end{align*}
where we recall that $K([\fl,\fl]; \la|_{\h\cap[\fl,\fl]},\zeta)$ is the irreducible Kostant Whittaker module of $[\fl,\fl]$,  $L([\fk,\fk];\la|_{\h\cap[\fk,\fk]})$ is the irreducible $[\fk,\fk]$-module of highest weight $\la|_{\h\cap[\fk,\fk]}$, and $\C_{\la^\perp}$ is the one-dimensional module of $\mf z$ corresponding to the character $\la^\perp$. This $\mf L$-module extends trivially to a $\mf P$-module so that we can form the {\it parabolic standard Whittaker module}
\begin{align}\label{def:Nzeta}
    N(\la,\zeta):=\text{Ind}^\g_{\mf P} K([\fl,\fl];\la|_{\h\cap[\fl,\fl]},\zeta)\otimes L([\fk,\fk]; \la|_{\h\cap[\fk,\fk]})\otimes\C_{\la^\perp}.
\end{align}
Then $N(\la,\zeta)\in\MS(\zeta)^{\mf q}$. As $N(\la,\zeta)$ is a quotient of $M(\la,\zeta)$, $N(\la,\zeta)$ has a unique irreducible quotient isomorphic to $L(\la,\zeta)$.

Let $\mc O^{\mf q}$ denote the parabolic subcategory of $\mc O$ consisting of $\g$-modules which are locally finite over $U(\mf n)$ and $\fk$-semisimple, and let $N(\la)$ denote the parabolic Verma module for $\la\in\h^\ast$ with $\la|_{\h\cap[\fk,\fk]}$ dominant integral. The following is the parabolic analogue of Proposition \ref{prop:gamma:O}.

\begin{prop}\label{prop:gammaq}
  The functor $\Gamma_\zeta:\mc O\rightarrow \MS(\zeta)$ restricts to an exact functor:
\begin{align*}
    \Gamma_\zeta:\mc O^{\mf q}\longrightarrow \MS(\zeta)^{\mf q}.
\end{align*}
Furthermore, for $\la\in\h^\ast$ such that $\la|_{\h\cap[\fk,\fk]}$ is dominant integral, we have
   \begin{align*}
   &\Gamma_\zeta\left(N(\la)\right)=N(\la,\zeta),\\
    &\Gamma_\zeta\left(L(\la)\right)=
       \begin{cases}
           L(\la,\zeta),\text{ if }\la\text{ is }W_{\fl}\text{-antidominant}\\
           0, \text{ otherwise}.
       \end{cases}
   \end{align*}
\end{prop}

\begin{proof}
For the first statement it suffices to show that for any $M\in\mc O^{\mf q}$, $\Gamma_\zeta(M)\in\MS(\zeta)^{\mf q}$.
    Since $\zeta$ vanishes on $[\fk,\fk]\cap \n$, from the construction of the Backelin functor in \cite[Section 3]{Ba97} we see that if $M$ is a direct sum of finite-dimensional modules over $\mf k$, then so is its image under the Backelin functor.

     As an $\mf l$-module we have
     \[
     N(\la)\cong U(\mf u_-)\otimes M([\fl,\fl]+\mf z;\la|_{(\h\cap[\fl,\fl])+\mf z})\otimes L([\fk,\fk];\la|_{\h\cap[\fk,\fk]}).
     \]
     Since the Backelin and the restriction functors commute, for $\la\in\h^\ast$, dominant integral on $\h\cap[\fk,\fk]$, we have:
    \begin{align*}
         \text{Res}^\g_{\mf l}\Gamma_\zeta N(\la)
         &\cong \Gamma_\zeta \text{Res}^\g_{\mf l} N(\la) \\
         &\cong \Gamma_\zeta\left( U(\mf u_-)\otimes M([\fl,\fl]+\mf z;\la|_{(\h\cap[\fl,\fl])+\mf z})\otimes L([\fk,\fk];\la|_{\h\cap[\fk,\fk]})\right)\\
        &\cong U(\mf u_-)\otimes \Gamma_\zeta\left( M([\fl,\fl]+\mf z;\la|_{(\h\cap[\fl,\fl])+\mf z})\right)\otimes L([\fk,\fk];\la|_{\h\cap[\fk,\fk]})\\
        &\cong U(\mf u_-)\otimes K([\fl,\fl]+\mf z;\la|_{(\h\cap[\fl,\fl])+\mf z},\zeta)\otimes L([\fk,\fk];\la|_{\h\cap[\fk,\fk]}).
    \end{align*}
    In the penultimate isomorphism above, we have use the fact that the Backelin functor commutes with tensoring with finite-dimensional modules, and that $L(\la|_{\h\cap[\fk,\fk]})$ is a finite-dimensional $[\fk,\fk]$-module and $\zeta$ is trivial on the simple roots of $[\fk,\fk]$. Thus, as $\mf l$-modules we have $N(\la,\zeta)\cong \Gamma_\zeta(N(\la))$. On the other hand, the above calculation shows that $\Gamma_\zeta(N(\la))$ has an $\mf l$-submodule  $K([\fl,\fl]+\mf z;\la|_{(\h\cap[\fl,\fl])+\mf z},\zeta)\otimes L([\fk,\fk];\la|_{\h\cap[\fk,\fk]})$ that is annihilated by $\mf u$ since it does not have any positive $\theta$-eigenvalue. Thus, we obtain, by Frobenius reciprocity, a non-zero homomorphism from $N(\la,\zeta)$ to $\Gamma_\zeta(N(\la))$, which maps surjectively onto this $\mf l$-module. Now, it is known that $\Gamma_\zeta(M(\la))$ is generated by its top, and hence so is $\Gamma_\zeta(N(\la))$. Hence this $\g$-homomorphism is surjective. As $N(\la,\zeta)$ and $\Gamma_\zeta(N(\la))$ are isomorphic as $\mf l$-modules, we conclude that they are also isomorphic as $\g$-modules.

    The last statement on $\Gamma_\zeta(L(\la))$ follows by Proposition \ref{prop:gamma:O}.
\end{proof}

\section{Parabolic cokernel categories}\label{sec:para:coker}

In this section, we formulate a ``parabolic cokernel subcategory" of the parabolic BGG category $\cO_\Z^{\mf q}$ of integral weight $\g$-modules and show it is a quotient category of $\cO^{\mf q}_\Z$. This parabolic cokernel subcategory admits favorable homological property known as properly stratified structure.

\subsection{Category $\mc O^{\mf q,\zeta\text{-pres}}$}
\label{sec:cat:qzeta}

Recall the triangular decomposition $\g=\n^- \oplus \h \oplus \n$ \eqref{eq:tri} and a character $\zeta:\mf n\rightarrow\C$ associated with the Levi subalgebra $\mf l$ (determined by an integral element $\theta \in \h$). We let $\mc O_\Z\subset \mc O$ be the full subcategory of integral weight $\g$-modules.

As in \S\ref{subsec:parabMMS}, we take a second even Levi subalgebra $\fk$ such that the two semisimple Lie subalgebras $[\fl,\fl]$ and $[\fk,\fk]$ commute with each other. We let $\mc O_\Z^{\mf q}\subset \mc O^{\mf q}$ be the full subcategory of integral weight $\g$-modules; recall $\mf q =\fk +\n$ from \eqref{eq:parab q}.

Denote by $\Lambda$ the set of integral weights in $\h^*$. Denote
\begin{align}
    \label{eq:Lzetaq}
    \begin{split}
    \Lambda(\zeta) &:=\{ \la \in \Lambda \mid \la \text{ is } W_{\fl}\text{-antidominant}\},
    \\
    \Lambda(\zeta)^{\mf q}  &:=\{ \la \in \Lambda(\zeta) \mid \la \text{ is dominant on } \h\cap[\fk,\fk] \}.
    \end{split}
\end{align}

A projective module in $\mc O_\Z$ (respectively, in $\mc O_\Z^{\mf q}$) is said to be $\zeta$-admissible, if it is a direct sum of projective covers of simple objects of highest weights in $\Lambda(\zeta)$ (respectively, in $\Lambda(\zeta)^{\mf q}$); see \eqref{eq:Lzetaq}.

Define the cokernel subcategory $\mc O^{\zeta\text{-pres}}$ of $\mc O_\Z$ (respectively, $\mc O^{\mf q,\zeta\text{-pres}}$ of $\mc O_\Z^{\mf q}$) to be the full subcategory consisting of objects $M$ such that there exists an exact sequence of the form
\begin{align*}
Q\longrightarrow P \longrightarrow M\longrightarrow 0,
\end{align*}
such that $Q$ and $P$ are $\zeta$-admissible projective modules in $\mc O_\Z$ (respectively, in $\mc O_\Z^{\mf q}$) (c.f., e.g., \cite[Section 2.3]{MaS05}).

Let $\mc I_\zeta$ denote the Serre subcategory of $\mc O_\Z$ generated by simple objects of the form $L(\la)$ with $\la\in\Lambda\setminus\Lambda(\zeta)$. Similarly, let $\mc I_\zeta^{\mf q}$ be the Serre subcategory of $\mc O_\Z^{\mf q}$ generated by simple objects of the form $L(\la)$ with $\la\in\Lambda\setminus\Lambda(\zeta)^{\mf q}$ such that $\la$ is dominant on $\h\cap[\fk,\fk]$.

Associated to $\mc I_\zeta$ and $\mc I^{\mf q}_\zeta$ we have the corresponding quotient categories $\mc O_\Z/\mc I_\zeta$ and $\mc O_\Z^{\mf q}/\mc I^{\mf q}_\zeta$ and quotient functors
\begin{align*}
    '\pi:\mc O_\Z\longrightarrow \mc O_\Z/\mc I_\zeta,\qquad
    \pi:\mc O_\Z^{\mf q}\longrightarrow \mc O_\Z^{\mf q}/\mc I^{\mf q}_\zeta.
\end{align*}
Since $\mc O_\Z^{\mf q}$ is a full abelian subcategory of $\mc O_\Z$ with compatible abelian structure, and $\mc I_\zeta^{\mf q}=\mc I_\zeta\cap \mc O_\Z^{\mf q}$, we conclude, from the definitions of quotient category and quotient functor, that $'\pi|_{\mc O_\Z^{\mf q}}=\pi$.

In \cite[Lemma 12]{CCM23} it was proved that the restriction of the functor $'\pi:\mc O_\Z\rightarrow \mc O_\Z/\mc I_\zeta$ to $\mc O^{\zeta\text{-pres}}$ gives an equivalence of categories $\mc O^{\zeta\text{-pres}}\cong\mc O_\Z/\mc I_\zeta$. The arguments therein can be adapted to prove the following parabolic analogue.

\begin{prop}\label{Opres=quot:gen}
    The restriction of the quotient functor $\pi:\mc O_\Z^{\mf q}\rightarrow \mc O_\Z^{\mf q}/\mc I^{\mf q}_\zeta$ to the subcategory $\mc O^{\mf q,\zeta{\text-pres}}$ gives an equivalence of categories:
    \begin{align*}
        \pi^{\mf q}:\mc O^{\mf q,\zeta\text{-pres}}\stackrel{\cong}{\longrightarrow} \mc O_\Z^{\mf q}/\mc I^{\mf q}_\zeta.
    \end{align*}
\end{prop}

\subsection{Properly stratified structure on $\mc O^{\mf q,\zeta\text{-pres}}$}\label{sec:Ozeta:para}

When $\g$ is a Lie algebra, it is known that the category $\mc O^{\zeta\text{-pres}}$ is a properly stratified category (cf. \cite{MaS05}). In the case when $\g$ is a basic classical Lie superalgebra the category $\mc O^{\zeta\text{-pres}}$ is properly stratified as well according to \cite[Section 5]{CCM23} and \cite[Section 4]{CC24}.

The arguments in \cite[Section 2]{MaS05} can be adapted to prove that the parabolic subcategory $\mc O^{\mf q,\zeta\text{-pres}}$ of $\mc O^{\zeta\text{-pres}}$ is properly stratified in the case when $\g$ is a Lie algebra. Using this, we can then apply the arguments in \cite{CCM23,CC24} to prove that the parabolic subcategory $\mc O^{\mf q,\zeta\text{-pres}}$ is properly stratified in the case when $\g$ is a basic classical Lie superalgebra as well. Below we shall give more precise statements.

For $\la \in \Lambda^{\mf q}(\zeta)$ let $P(\g,\la)\in \mc O_\Z^{\mf q}$ denote the projective cover of $L(\g,\la)$. We shall write $P(\la)$ for $P(\g,\la)$ when $\g$ is clear from the context. We define the following module in $\mc O^{\mf q,{\zeta}\text{-pres}}$:
\begin{align*}
    S(\la):=P(\la)/(\text{rad}P(\la)^{\texttt{tr}}),
\end{align*}
where we have denoted by $M^{\texttt{tr}}$ the sum of all homomorphic images of $\zeta$-admissible projective modules in $\mc O_\Z^{\mf q}$ to a module $M$ in $\mc O_\Z^{\mf q}$.
 Furthermore, we define the following standard and proper standard modules, respectively:
\begin{align*}
    \Delta(\la) &=\text{Ind}^\g_{\mf P}P(\mf L,\la),\\
    \blacktriangle(\la)&=P(\la)/(Q(\la)^{\texttt{tr}}),
\end{align*}
where $P(\mf L,\la)\cong P([\fl,\fl]\oplus\mf z,\la|_{\h\cap[\fl,\fl]+\mf z})\otimes L([\fk,\fk],\la|_{\h\cap[\fk,\fk]})$ is the projective cover of the irreducible $\mf L$-module $L(\mf L,\la)$ in the corresponding parabolic category of $\mf L$-modules, and $Q(\la)$ is the kernel of the canonical map $P(\la)\rightarrow N(\la)\rightarrow 0$. Both $\Delta(\la)$ and $\blacktriangle(\la)$ lie in $\mc O^{\mf q,\zeta\text{-pres}}$. The arguments in \cite[Section 2]{MaS05} and \cite[Section 4]{CC24} can be adapted now to show that $\mc O^{\mf q,\zeta\text{-pres}}$ is properly stratified with indecomposable projective objects $P(\la)$, standard and proper standard objects $\Delta(\la)$ and $\blacktriangle(\la)$, respectively, and simple objects $S(\la)$, where $\la\in\Lambda(\zeta)^{\mf q}$. In particular, $P(\la)$ has a filtration, subquotients of which are $\Delta(\mu)$ with $\mu\succeq\la$. Also, $\Delta(\la)$ has a filtration of length $|W_{\fl}\cdot\la|$ with each subquotient isomorphic to $\blacktriangle(\la)$. We summarize the above discussion in the following.

\begin{prop}\label{prop:para:strat:g}
    The category $\mc O^{\mf q,\zeta\text{-pres}}$ is a properly stratified category with indecomposable projective, standard, proper standard, and simple objects $P(\la)$, $\Delta(\la)$, $\blacktriangle(\la)$ and $S(\la)$, for $\la\in\Lambda(\zeta)^{\mf q}$, respectively. Furthermore, the following BGG-type reciprocity holds in $\mc O^{\mf q,\zeta\text{-pres}}$:
    \begin{align*}
        (P(\la):\Delta(\mu)) = [\blacktriangle(\mu):S(\la)], \qquad \text{ for } \la,\mu\in\Lambda(\zeta)^{\mf q}.
    \end{align*}
\end{prop}

\subsection{Tilting modules in $\mc O^{\mf q,\zeta\text{-pres}}$}

It is well known that there exists a duality functor ${\cdot}^\vee$ on $\mc O_\Z$ which restricts to a simple-preserving duality functor on $\mc O_\Z^{\mf q}$. We note that for $\la\in\Lambda(\zeta)^{\mf q}$, the projective module $P(\la)$, as an $\mf L$-module, is a direct sum of self-dual projective modules. Now, we can show, following the arguments in \cite[Propositions 2.8 and 2.9]{MaS05}, that if $M\in\mc O_\Z^{\mf q}$, such that, as an $\mf L$-module, $M$ is a direct sum of self-dual projective modules, then $M\in\mc O^{\mf q,\zeta\text{-pres}}$. This in particular implies that
\begin{align*}
    \nabla(\la):=\Delta(\la)^\vee \in\mc O^{\mf q,\zeta\text{-pres}}.
\end{align*}

Let $\mc F(\Delta)$ be the full subcategory of $\mc O^{\mf q,\zeta\text{-pres}}$ of modules with finite $\Delta$-flags. We define similarly $\mc F(\nabla)$ to be the full subcategory of $\mc O^{\mf q,\zeta\text{-pres}}$ of modules that have finite $\nabla$-flags. A module $T\in\mc O^{\mf q,\zeta\text{-pres}}$ is called a tilting module if $T\in\mc F(\Delta)\cap\mc F(\nabla)$. Any tilting module is a direct sum of indecomposable tilting modules, and the indecomposable tilting modules in $\mc O^{\mf q,\zeta\text{-pres}}$ are parametrized by $\Lambda(\zeta)^{\mf q}$. We denote the indecomposable tilting module corresponding to $\la\in\Lambda(\zeta)^{\mf q}$ by $T(\la)$ with $\la$ as its highest weight. Denote the indecomposable tilting module in $\mc O_\Z^{\mf q}$ of highest weight $\la\in \Lambda$, dominant on $\h\cap[\fk,\fk]$, by $T^{\mc O_\Z^{\mf q}}(\la)$. Denote by $w_0^{\mf k}$ (respectively, $w_0^{\fl}$) the longest element in the Weyl group of $\mf k$ (respectively, $\mf l$).

\begin{prop}\label{prop:para:stratified}
    Let $\la\in\Lambda(\zeta)^{\mf q}$. Then
    \begin{align*}
        T(\la)= T^{\mc O_\Z^{\mf q}}(w_0^{\fl}\cdot\la).
    \end{align*}
Furthermore, the following Ringel-type duality holds:
\begin{align*}
        (T(\la):\Delta(\mu)) = [\blacktriangle(-w_0^{\fl} w_0^{\mf k}\cdot\mu-\rho):S(-w_0^{\fl} w_0^{\mf k}\cdot\la-\rho)].
    \end{align*}
\end{prop}

\begin{proof}
We shall first establish the proposition for $\g$ a Lie algebra. To do that we will prove that tilting modules in $\mc O_\Z^{\mf q}$ of the form $T^{\mc O_\Z^{\mf q}}(w_0^{\fl}\cdot\la)$ indeed have $\Delta$-flags by adapting the arguments in \cite[Lemma 18]{FKM00}.

Let $\nu\in\Lambda$ be an anti-dominant weight that is regular on the $[\fk,\fk]$.  Let $w_0^{\mf k}$ be the longest element in Weyl group of the subalgebra $[\fk,\fk]$. We observe that the parabolic Verma module $N(w_0^{\mf k}\cdot\nu)$ is irreducible (see, e.g., \cite[Lemma 2]{Ja77}). Thus, in particular, $L(w_0^{\mf k}\cdot\nu)$ is the socle of a parabolic Verma module, and hence, by a classical theorem of Irving \cite{Ir85}, its projective cover $P(w_0^{\mf k}\cdot\nu)$ is self-dual and hence a tilting module. On the other hand, we have that $P(w_0^{\mf k}\cdot\nu)=\Delta(w_0^{\mf k}\cdot\nu)$ by the same argument as \cite[Proposition 2.9(i)]{MaS05}. Thus, we conclude that $\Delta(w_0^{\mf k}\cdot\nu)\cong T^{O_\Z^{\mf q}}(w_0^{\fl} w_0^{\mf k}\cdot\nu)$ and $T^{O_\Z^{\mf q}}(w_0^{\fl} w_0^{\mf k}\cdot\nu)$ has a $\Delta$- and a $\nabla$-flag.

One sees that every other tilting module $T^{O_\Z^{\mf q}}(\la)$, with $\la\in\Lambda(\zeta)^{\mf q}$, can be obtained as a direct summand of the tensor product of such a tilting module as above with a finite-dimensional module of $\g$. Since they have $\Delta$- and $\nabla$-flags, they are indeed tilting module in $\mc O^{\mf q,\zeta\text{-pres}}$. From this it is clear that $T^{\mc O_\Z^{\mf q}}(w_0^{\fl}\cdot\la)$ contains $\Delta(\la)$ as a submodule, and hence, by a standard uniqueness argument,  is isomorphic to $T(\la)$.

The Ringel-type duality in the second statement is now a consequence of the Ringel duality for the parabolic category $\mc O_\Z^{\mf q}$ (see, e.g, \cite[Corollary 3.8]{CCC21}) and the fact that $(T^{\mc O_\Z^{\mf q}}(w_0^{\fl}\cdot\la):N(\mu))=(T^{\mc O_\Z^{\mf q}}(w_0^{\fl}\cdot\la):N(w\cdot\mu))$, for any $w\in W_{\fl}$. The reader is referred to \cite[Corollary 18]{CC24} for more details.

Now suppose that $\g$ is a basic classical Lie superalgebra that is not a Lie algebra. The existence of tilting modules in ${\mc O}^{\mf q,\zeta\text{-pres}}$ can now be derived using the existence of tilting modules in the parabolic cokernel category for the Lie algebra $\g_{\bar 0}$ following the arguments in \cite[\S4.4]{CC24}.
\end{proof}

\subsection{Properly stratified Whittaker categories}\label{sec:strat:Whitt}

Let $\g$ be a basic classical Lie superalgebra with $\zeta:\mf n\rightarrow\C$ as in \eqref{eq:zeta}. Let $\mc W(\zeta) \subset \MS(\zeta)$ be the image category of the Backelin functor $\Gamma_\zeta:\mc O_\Z\rightarrow \MS(\zeta)$, where we keep the same notation to denote the restriction of the Backelin functor to $\mc O_\Z$. It was shown in \cite[Corollary 38]{CCM23} and \cite[Corollary 28]{CC24} that the functor $\Gamma_\zeta:\mc O_\Z\rightarrow\mc W(\zeta)$ satisfies the universal property of the Serre quotient functor corresponding to the Serre subcategory $\mc I_\zeta$, and hence it induces an equivalence of categories
\begin{align}
   {}'{\Gamma}_\zeta:\mc O_\Z/\mc I_\zeta\stackrel{\cong}{\longrightarrow} \mc W(\zeta).
\end{align}
Now, $\mc W(\zeta)$ contains the standard Whittaker modules and also all simple modules of integral central characters in $\MS(\zeta)$. Furthermore, since $\mc O^{\zeta\text{-pres}}$ is properly stratified by Proposition~\ref{prop:para:strat:g}, it follows that the category $\mc W(\zeta)$ is also properly stratified. In this section we shall give a parabolic analogue of this result.

Let $\mf q$ is another parabolic subalgebra of $\g$ with even Levi subalgebra $\mf k$ satisfying compatibility condition with $\mf p$ as in previous sections. We consider the restriction of the Backelin functor $\Gamma_\zeta:\mc O_\Z\rightarrow \mc W(\zeta)$ to the parabolic subcategory $\mc O_\Z^{\mf q}$. Denote the corresponding image subcategory of $\mc W(\zeta)$ by $\mc W(\zeta)^{\mf q}$ so that we have a functor $\Gamma^{\mf q}_\zeta:\mc O_\Z^{\mf q}\rightarrow\mc W(\zeta)^{\mf q}$.

\begin{prop}\label{prop:gamma:para}
    The functor ${\Gamma}^{\mf q}_\zeta:\mc O_\Z^{\mf q}\rightarrow \mc W(\zeta)^{\mf q}$ induces an equivalence of categories
    \begin{align*}
        {}'{\Gamma}^{\mf q}_\zeta:\mc O_\Z^{\mf q}/\mc I^{\mf q}_\zeta\stackrel{\cong}{\longrightarrow} \mc W(\zeta)^{\mf q},
    \end{align*}
    and the category $\mc W(\zeta)^{\mf q}$ inherits a properly stratified structure.
    Furthermore, $\mc W(\zeta)^{\mf q}$ contains $N(\la,\zeta)$ and simple Whittaker module $L(\la,\zeta)$, for $\la\in\Lambda(\zeta)^{\mf q}$.
\end{prop}

\begin{proof}
The functor $\Gamma_\zeta$ vanishes on objects in $I_\zeta$ and hence vanishes on objects in $I_\zeta^{\mf q}$. Indeed, since we have $\mc I_\zeta\cap \mc O_\Z^{\mf q}=\mc I_\zeta^{\mf q}$, we conclude that $\Gamma_\zeta^{\mf q}$ vanishes on $\mc I^{\mf q}_\zeta$. Thus, by the universal property of Serre quotient, \cite[Corollaires III.1.2 and III.1.3]{Gabriel}, the restriction $\Gamma_\zeta^{\mf q}$ induces a quotient functor ${}'{\Gamma}^{\mf q}_\zeta:\mc O_\Z^{\mf q}/\mc I_\zeta^{\mf q}\rightarrow \mc W(\zeta)^{\mf q}$. Since $\mc O_\Z^{\mf q}$ is a full abelian subcategory of $\mc O_\Z$ with compatible abelian structure, by definition of the morphisms in a quotient category (see, e.g., \cite[\S4.2]{CCM23}), we see that $\mc O_\Z^{\mf q}/\mc I^{\mf q}_\zeta$ is a full subcategory of $\mc O_\Z/\mc I_\zeta$ and ${}'{\Gamma}^{\mf q}_\zeta={}'{\Gamma}_\zeta|_{\mc O_\Z^{\mf q}/\mc I^{\mf q}_\zeta}$. Now, we have that ${}'{\Gamma}_\zeta:\mc O_\Z/\mc I_\zeta\rightarrow W(\zeta)$ is full and faithful by \cite[Theorem 37]{CCM23} and \cite[Corollary 29]{CC24}. It follows that ${}'{\Gamma}_\zeta|_{\mc O_\Z^{\mf q}/\mc I^{\mf q}_\zeta}$ is full and faithful. Be definition, it is essentially surjective. This proves that ${}'{\Gamma}^{\mf q}_\zeta:\mc O_\Z^{\mf q}/\mc I^{\mf q}_\zeta \rightarrow \mc W(\zeta)^{\mf q}$ is an equivalence.

By Propositions \ref{Opres=quot:gen} and \ref{prop:para:strat:g}, we have $\mc O^{\mf q,\zeta\text{-pres}}\cong \mc O_\Z^{\mf q}/\mc I^{\mf q}_\zeta$ and they are properly stratified. Thus $\mc W(\zeta)^{\mf q}$ is also properly stratified by the equivalence ${}'{\Gamma}^{\mf q}_\zeta.$

Now by Proposition~ \ref{prop:gammaq}, the modules $N(\la,\zeta)$ and their unique irreducible quotients $L(\la,\zeta)$ lie in the category $\mc W(\zeta)^{\mf q}$, for $\la\in\Lambda(\zeta)^{\mf q}$.
\end{proof}

\section{Super duality for Whittaker modules}\label{sec:SD:MMS}

In this section we formulate Lie superalgebras $\sg_n$ and  Lie algebras $\g_n$ of classical type, for $0\le n \le \infty$. We apply the results in the prior sections to obtain an equivalence of certain parabolic categories between Whittaker modules over $\g_n$ and $\sg_n$ as $n$ tends to infinity.

\subsection{Super duality for parabolic BGG categories}
\label{sec:SD1}

We shall recall the setup from \cite[\S6.1]{CW12} below.

We have head diagrams \framebox{$\mf H^{\mf x}$}, where $\mf x=\mf{a,b,c,d}$, representing Dynkin diagrams of simple Lie algebras of type $A,B,C,D$, respectively, and tail diagrams \framebox{$\mf T_n$} and \framebox{$\ov{\mf T}_n$}, representing Dynkin digrams of Lie superalgebras $\gl(1+n)$ and $\gl(1|n)$, respectively. Connecting the type A end of the head diagram \framebox{$\mf H^{\mf x}$} with the first vertex of the tail diagrams \framebox{$\mf T_n$} and \framebox{$\ov{\mf T}_n$}, we get Dynkin diagrams of Lie algebras and Lie superalgebras of type $A,B,C,D$.  (We remark that our \framebox{$\mf H^{\mf x}$} here is denoted by \framebox{$\mf{k^x}$} in \cite[Chapter 6]{CW12}.)

We denote by $(\g_n,\sg_n)$ a pair consisting of a Lie algebra $\g_n$ and a Lie superalgebra $\sg_n$, where $\g_n$ is the Lie algebra corresponding the Dynkin diagram \framebox{$\mf H^{\mf x}$}---\framebox{${\mf T}_n$} and $\sg_n$ is the Lie superalgebra corresponding the Dynkin diagram \framebox{$\mf H^{\mf x}$}---
\framebox{$\ov{\mf T}_n$}, respectively.

\begin{ex}
\begin{enumerate}
    \item
Let \framebox{$\mf H^{\mf a}$} be the Dynkin diagram for $\mf{gl}(m)$. Then we have $\g_n\cong\mf{gl}(m+n)$ and $\sg_n\cong\mf{gl}(m|n)$.
\item
Let \framebox{$\mf H^{\mf b}$} be the Dynkin diagram for $\mf{so}(2m+1)$. Then we have $\g_n\cong\mf{so}(2m+2n+1)$ and $\sg_n\cong\mf{osp}(2m+1|2n)$.
\end{enumerate}
\end{ex}

We have natural embeddings $\g_n \subset \g_{n+1}$ and $\sg_n \subset \sg_{n+1}$, and hence the respective direct limits $\g_\infty$ and $\sg_\infty$ are valid. In the cases of $\mf{x}=\mf{b,c,d}$ we recall that to deal with $\g_\infty$ and $\sg_\infty$ and truncations to $\g_n$ and $\sg_n$ for finite $n$, it is more convenient and conceptual to introduce trivial central extensions of $\g_n$ and $\sg_n$ by a one-dimensional central element and work with these instead. But we shall ignore this issue as much as possible to keep subsequent notation and presentation simpler, the reader is referred to \cite[\S6.1.6]{CW12} for the precise details. We let $\mc O_n$ and $\ov{\mc O}_n$ denote the corresponding BGG categories of $\g_n$- and $\sg_n$-modules, respectively. Denote the corresponding Verma modules of highest weight $\la\in\h^*$ by $M_n(\la)$ and $\ov{M}_n(\la)$, and their simple heads by $L_n(\la)$ and $\ov{L}_n(\la)$, respectively. Here, $\h$ denotes the Cartan subalgebra of either $\g_n$ or $\sg_n$.

Let $\mf k$ be the Levi subalgebra of $\sg_n$ with semisimple summand $[\gl(n),\gl(n)]$, where $\gl(n)\subset \gl(1|n)$ and $\gl(1|n)$ is the Lie superalgebra corresponding to the Dynkin diagram \framebox{$\ov{\mf T}_n$}. We also regard the same $\mf k$ as a Levi subalgebra of $\g_n$, where $\gl(n)\subset \gl(1+n)$ for $\gl(1+n)$ corresponding to the Dynkin diagram \framebox{$\mf T_n$}. We shall use $\ufk$ denote this copy of $\gl(n)$. Let $\mf b$ and $\ov{\mf b}$ denote the standard Borel subalgebras containing $\h$ corresponding to the Dynkin diagrams of $\g_n$ and $\sg_n$, respectively. Denote the associated parabolic subalgebras by $\mf q=\mf b+\mf k$ and $\ov{\mf q}=\ov{\mf b}+\mf k$, respectively. We let $\mc O^+_n$ and $\ov{\mc O}^+_n$ be the full subcategories of $\mc O_n^{\mf q}$ and $\ov{\mc O}_n^{\mf q}$, respectively, consisting of objects $M$ on which the action of $\ufk$ is polynomial. The corresponding parabolic Verma modules in $\mc O^+_n$ and $\ov{\mc O}^+_n$ of highest weight $\la\in\h^*$ are denoted by $N_n(\la)$ and $\ov{N}_n(\la)$, respectively. We note that if $\{E_{ii}|1\le i\le n\}$, is the standard basis for the Cartan subalgebra $\ufh$ of $\ufk$,
and $\{\ep_i|1\le i\le n\}$ is its dual basis, then $\la|_{\ufh}=\sum_{i=1}^n\la_i\ep_i$ for a partition $(\la_1,\ldots,\la_n)$. The subset of such weights in $\h^*$ will be denoted by $\h^{*,+}$. For $\la\in\h^{*,+}$, we let $\la^\natural\in\h^{*,+}$ be obtained from $\la$ by replacing $\la|_{\ufh}=\sum_{i=1}^n\la_i\ep_i$ by $\sum_{i=1}\la'_i\ep_i$, where $(\la'_1,\la'_2,\ldots)$ is the conjugate partition of $(\la_1,\la_2,\ldots)$. Furthermore, the categories $\mc O_n^+$ and $\ov{\mc O}_n^+$ have indecomposable tilting modules of highest weight $\la\in\h^{\ast,+}$, denoted by $T_n(\la)$ and $\ov{T}_n(\la)$, respectively.

The relationships between $\mc O_n^+$, for various $n$, between $\ov{\mc O}_n^+$, for various $n$, are given by the truncation functors, which we shall recall below. The reader is referred to \cite[\S6.2.5]{CW12} for more details.

Let $M\in\mc O_n^+$. The $\ufk$-module $\text{Res}^\g_{\ufk} M$ is a direct sum of finite-dimensional irreducible polynomial $\ufk$-modules. Thus we have a weight space decomposition:
\begin{align*}
    M=\sum_{\mu\in\h^*}M_\mu,
\end{align*}
such that $\mu|_{\ufh} =\sum_{i=1}^n{\mu_i}\ep_i$, with $\mu_i\in\mathbb N$. For $m< n$, we define
\begin{align*}
    \text{Tr}^{n}_m(M):=\sum_{\mu_j=0,\forall j>m} M_\mu.
\end{align*}
Then it is easy to see that $\text{Tr}^{n}_m(M)\in\mc O^+_m$.

\begin{prop}\label{prop:trun:O}
    For $1\le m< n\le\infty$, $\text{Tr}^{n}_m:\mc O^+_n\rightarrow\mc O^+_m$ defines an exact functor. Furthermore, for $X=N,L,T$, and $\la\in\h^{\ast,+}$ with $\la|_{\ufh}=\sum_{i=1}^n\la_i\ep_i$, we have
    \begin{align*}
        \text{Tr}^{n}_m(X_n(\la))=\begin{cases}
            X_m(\la), \text{ if }\la_{m+1}=0,\\
            0, \text{ otherwise.}
        \end{cases}
    \end{align*}
\end{prop}

We remark that we can define analogously an exact truncation functor $\ov{\text{Tr}}^{n}_m:\ov{\mc O}^+_n\rightarrow\ov{\mc O}^+_m$, for $n>m$, and have a super-analogue of Proposition~ \ref{prop:trun:O}.

\begin{prop}\label{prop:trun:SO}
    Let $1\le m <n\le\infty$. For $\ov{X}=\ov{N},\ov{L},\ov{T}$, and $\la\in\h^{\ast,+}$ with $\la|_{\ufh}=\sum_{i=1}^n\la_i\ep_i$, we have
    \begin{align*}
        \ov{\text{Tr}}^{n}_m(\ov{X}_n(\la))=\begin{cases}
            \ov{X}_m(\la), \text{ if }\la_{m+1}=0,\\
            0, \text{ otherwise.}
        \end{cases}
    \end{align*}
\end{prop}

A relationship between $\mc O^+_n$ and $\ov{\mc O}^+_n$ is given by the so-called super duality \cite[Chapter 6]{CW12}, which holds only at $n=\infty$. In the sequel we shall make it a convention of notation to sometimes drop the subscript $n$ when considering $n=\infty$.

\begin{thm}\label{thm:SD}\cite{CW08, CL10, CLW11}
We have an equivalence of the BGG categories for $\g_\infty$ and $\sg_\infty$:
\begin{align*}
    \mc O^+\cong\ov{\mc O}^+,
\end{align*}
under which $N(\la)$, $L(\la)$ and $T(\la)$ correspond to $\ov{N}(\la^\natural)$, $\ov{L}(\la^\natural)$ and $\ov{T}(\la^\natural)$, respectively, for $\la\in \h^{*,+}$.
\end{thm}

Let $\mf H$ denote the Lie subalgebra of $\g_n$ and $\sg_n$ whose Dynkin diagram is \framebox{$\mf H^{\mf x}$}. Let $\mf N$ be a nilradical of the standard Borel of $\mf H$ and $\zeta:\mf N\rightarrow\C$ be a character of $\mf N$. Let $\mf{n}_n$ and $\ov{\mf{n}}_n$ be the nilradicals of the standard Borel subalgebras of $\g_n$ and $\sg_n$, respectively. Thanks to $\mf N \subset \mf{n}_n$ and $\mf N \subset \ov{\mf{n}}_n$, we can extend $\zeta$ trivially to a character of $\n_n$ and $\ov{\n}_n$, respectively. We shall again denote them by $\zeta$ by abuse of notation.

Recall the MMS categories introduced in \S\ref{subsec:MMS}.
Denote the categories of MMS Whittaker modules of $\g_n$ and $\sg_n$ by $\MS(\zeta)_n$ and $\ov{\MS}(\zeta)_n$, respectively. Denote by $\MS(\zeta)_n^+$ (and respectively, $\ov{\MS}(\zeta)_n^+$) the full subcategory of $\MS(\zeta)_n$ (and respectively, $\ov{\MS}(\zeta)_n$) consisting of objects on which the $\gl(n)$-action is polynomial. We recall the construction of parabolic standard Whittaker modules from \eqref{def:Nzeta}. We denote the corresponding parabolic standard Whittaker modules of $\g_n$ and $\sg_n$ by $N(\la,\zeta)$ and $\ov{N}(\la,\zeta)$, respectively. Similarly, we have the self-explanatory notations of $L(\la,\zeta)$ and $\ov{L}(\la,\zeta)$.

We shall denote the Backelin functor for $\g_n$ and $\sg_n$ by $\Gamma_\zeta$ and
$\ov{\Gamma}_\zeta$, respectively.

\begin{lem}\label{lem:gammaq2}
  The functors $\Gamma_\zeta$ and $\ov{\Gamma}_\zeta$ restrict to the following functors:
\begin{align*}
    \Gamma^+_\zeta:\mc O_n^+\longrightarrow \MS(\zeta)_n^+,\qquad
    \ov{\Gamma}^+_\zeta: \ov{\mc O}^+_n\longrightarrow \ov{\MS}(\zeta)_n^+.
\end{align*}
Furthermore, for $\la\in\h^{\ast,+}$, we have
   \begin{align*}
   &\Gamma_\zeta\left(N(\la)\right)=N(\la,\zeta),\\
    &\Gamma_\zeta\left(L(\la)\right)=
       \begin{cases}
           L(\la,\zeta),\text{ if }\la\text{ is }W_{\fl}\text{-antidominant}\\
           0, \text{ otherwise};
       \end{cases} \\
    &\ov{\Gamma}_\zeta\left(\ov{N}(\la)\right)=\ov{N}(\la,\zeta),\\
    &\ov{\Gamma}_\zeta\left(\ov{L}(\la)\right)=
       \begin{cases}
           \ov{L}(\la,\zeta),\text{ if }\la\text{ is }W_{\fl}\text{-antidominant}\\
           0, \text{ otherwise}.
       \end{cases}
   \end{align*}
\end{lem}

\begin{proof}
    By Proposition \ref{prop:gammaq} if $M$ is a direct sum of finite-dimensional irreducible $\gl(n)$-modules, then so is $\Gamma^+_\zeta(M)$. The double dual construction in the proof of \cite[Lemma~ 3.2]{Ba97} also shows that if $M$ is polynomial, then so is the image under the Backelin functor.
\end{proof}

\subsection{Properly stratified categories $\mc O_n^{\zeta\text{-pres}+}$ and $\ov{\mc O}_n^{\zeta\text{-pres}+}$}

We shall now apply the results in \S\ref{sec:Ozeta:para} to our setting of $\g_n$- and $\sg_n$-modules. We let $\mc O_{n,\Z}\subset \mc O_n$ (and respectively, $\ov{\mc O}_{n,\Z}\subset \ov{\mc O}_n$ ) denote the full subcategory of integral weight $\g_n$-modules (and respectively, $\sg_n$-modules).

The set $\Lambda$ of integral weights will be denoted by $\Lambda_n$ and $\Lambda(\zeta)^{\mf q}$ from \eqref{eq:Lzetaq} will be denoted by $\Lambda(\zeta)^{\mf q}_n$ in our current setting for $\g_n$ and $\sg_n$. Denote further
\begin{align}
\label{L+zeta}
\begin{split}
    \Lambda_n^+ &:=\{\la\in \Lambda_n \mid \la|_{\ufh}\text{ is polynomial}\},
    \\
    \Lambda(\zeta)_n^+&:=\{\la\in \Lambda(\zeta)^{\mf q}_n \mid \la|_{\ufh}\text{ is polynomial}\}.
\end{split}
\end{align}

Recall from \S\ref{sec:cat:qzeta} the parabolic cokernel categories of $\g_n$- and $\sg_n$-modules, which we shall denote by $\mc O_n^{\mf q,\zeta\text{-pres}}$ and $\ov{\mc O}_n^{\ov{\mf q},\zeta\text{-pres}}$, respectively. We define subcategories
\begin{align*}
    \mc O_n^{\zeta\text{-pres}+}\subseteq \mc O_n^{\mf q,\zeta\text{-pres}}\quad\text{and}\quad \ov{\mc O}_n^{\zeta\text{-pres}+}\subseteq \ov{\mc O}_n^{\ov{\mf q},\zeta\text{-pres}}
\end{align*}
consisting of modules $M$ in the respective categories on which the $\gl(n)$-action is polynomial.

The same arguments in \S\ref{sec:cat:qzeta} gives us equivalences of categories
\begin{align*}
    \mc O_n^{\zeta\text{-pres}+} \cong \mc O^+_{n,\Z}/\mc I_\zeta^+,
    \qquad
    \ov{\mc O}_n^{\zeta\text{-pres}+} \cong \ov{\mc O}^+_{n,\Z}/\ov{\mc I}_\zeta^+,
\end{align*}
where $\mc O^+_{n,\Z}/\mc I_\zeta^+$ and $\ov{\mc O}^+_{n,\Z}/\ov{\mc I}_\zeta^+$ denote the Serre quotient categories by the respective Serre subcategories $\mc I_\zeta^+=\mc I^{\mf q}_\zeta\cap\mc O_{n,\Z}^+$ and $\ov{\mc I}_\zeta^+=\ov{\mc I}^{\ov{\mf q}}_\zeta\cap\ov{\mc O}_{n,\Z}^+$.

Furthermore, the arguments in \S\ref{sec:Ozeta:para} show that $\mc O_n^{\zeta\text{-pres}+}$ and $\ov{\mc O}_n^{\zeta\text{-pres}+}$
are both properly stratified, with standard, proper standard, projective and simple objects given by the same $\Delta(\la)$, $\blacktriangle(\la)$, $P(\la)$ and $S(\la)$ as in \S\ref{sec:Ozeta:para}, now with $\la\in\Lambda(\zeta)_n^+$ in \eqref{L+zeta}. To continue to distinguish between $\g_n$- and $\sg_n$-modules, we shall denote the $\sg_n$-counterparts by $\ov{\Delta}(\la)$, $\ov{\blacktriangle}(\la)$, $\ov{P}(\la)$ and $\ov{S}(\la)$ accordingly. We remark that the existence of projective covers in these categories in the limit $n\to\infty$ is established in \cite[Theorem 3.5]{CL20} and \cite[Theorem 3.12]{Le20}. We summarize the above discussion in the following.

\begin{prop}\label{prop:para:strat}
    For $n\leq \infty$, the categories $\mc O^{\zeta\text{-pres}+}_n$ and $\ov{\mc O}^{\zeta\text{-pres}+}_n$ are categories with properly stratified structures with indecomposable projective, standard, proper standard, and simple objects described above. Furthermore, for $\la,\mu\in\Lambda(\zeta)_n^+$, we have the following BGG-type reciprocity:
    \begin{align*}
        (P(\la):\Delta(\mu) = [\blacktriangle(\mu):S(\la)],\\
        (\ov{P}(\la):\ov{\Delta}(\mu) = [\ov{\blacktriangle}(\mu):\ov{S}(\la)].
    \end{align*}
\end{prop}

We denote the indecomposable tilting module corresponding to $\la\in\Lambda(\zeta)_n^+$ in $\mc O^{\zeta\text{-pres}+}_n$ by $T(\la)$. Denote the indecomposable tilting module in $\mc O_{n,\Z}^+$ of highest weight $\la\in \Lambda^+$ by $T^{\mc O_\Z^+}(\la)$. Similarly, we have the notations $\ov{T}(\la)$ and $\ov{T}^{\ov{\mc O}_\Z^+}(\la)$.

\begin{prop} Let $\la,\mu\in\Lambda(\zeta)_n^+$. For $n\le\infty$, we have
    \begin{align*}
        T(\la)\cong T^{\mc O_\Z^+}(w_0^{\fl}\cdot\la)\quad\text{and}\quad \ov{T}(\la)\cong \ov{T}^{\ov{\mc O}_\Z^+}(w_0^{\fl}\cdot\la).
    \end{align*}
Furthermore, for $n<\infty$, the following Ringel-type dualities hold:
\begin{align*}
        (T(\la):\Delta(\mu)) = [\blacktriangle(-w_0^{\fl} w_0^n\cdot\mu-\rho):S(-w_0^{\fl} w_0^n\cdot\la-\rho)],\\
        (\ov{T}(\la):\ov{\Delta}(\mu)) = [\ov{\blacktriangle}(-w_0^{\fl} w_0^n\cdot\mu-\rho):\ov{S}(-w_0^{\fl} w_0^n\cdot\la-\rho)],
    \end{align*}
    where $w_0^n$ is the longest element in Weyl group of the subalgebra $\gl(n)$.
\end{prop}

\begin{proof}
The case $n<\infty$ is a direct consequence of Proposition \ref{prop:para:stratified}. So it remains to prove the first statement for $n=\infty$.

For $n=\infty$, we observe that, by Proposition \ref{prop:trun:O}, the tilting modules of the form $T^{\mc O_\Z^+}(w_0^{\fl}\cdot\la)$ truncate to the corresponding tilting modules of the same highest weight for $n\gg0$ having the same parabolic Verma flag length. Furthermore, since the $\Delta$- and $\nabla$-flags are also compatible under the truncation functors, these tilting modules have $\Delta$- and $\nabla$-flags as well, and thus, they are indeed tilting modules $T(\la)$ for $\mc O^{\zeta\text{-pres}+}$ at $n=\infty$.

Similarly, we have the identity for $\ov{T}(\la)\cong \ov{T}^{\ov{\mc O}_\Z^+}(w_0^{\fl}\cdot\la)$ for $n=\infty$ using Proposition \ref{prop:trun:SO}.
\end{proof}

\subsection{Categories $\calW(\zeta)_n^+$ and $\ov{\calW}(\zeta)_n^+$ of Whittaker modules}

Following \S\ref{sec:strat:Whitt},  we let $\calW(\zeta)_n^+$ and $\ov{\calW}(\zeta)_n^+$ denote the image categories of $\Gamma^+_\zeta:\mc O_{n,\Z}^+\rightarrow\MS(\zeta)_n^+$ and $\ov{\Gamma}^+_\zeta:\ov{\mc O}_{n,\Z}^+\rightarrow\ov{\MS}(\zeta)_n^+$, respectively.

\begin{prop}\label{prop:gamma+g}
    Let $n\le\infty$. The functor ${\Gamma}_\zeta^+:\mc O_{n,\Z}^+\rightarrow \mc W(\zeta)_n^+$ induces an equivalence of categories
    \begin{align*}
        {}'{\Gamma}_\zeta^+:\mc O_{n,\Z}^+/\mc I^+_\zeta\stackrel{\cong}{\longrightarrow} \mc W(\zeta)_n^+.
    \end{align*}
    Furthermore, $\Gamma_\zeta(N(\la)) =N(\la,\zeta)$, for $\la\in \Lambda_n^+$, and
   \begin{align*}
       \Gamma_\zeta\left(L(\la)\right)=
       \begin{cases}
           L(\la,\zeta),\text{ if }\la\in\Lambda(\zeta)_n^+;\\
           0, \text{ otherwise}.
       \end{cases}
   \end{align*}
    Hence $\mc W(\zeta)_n^+$ contains $N(\la,\zeta)$ and the simple Whittaker modules $L(\la,\zeta)$, for $\la\in\Lam(\zeta)_n^+$.
\end{prop}

\begin{proof}
    The case of $n<\infty$ is a direct consequence of Proposition \ref{prop:gamma:para}

   Now suppose that $n=\infty$. The same type of arguments as Proposition \ref{prop:gammaq} also shows that the Backelin functor sends the parabolic Verma and irreducible modules in $\mc O^+$ to parabolic standard Whittaker and irreducible modules, respectively, in the case as well. So, it remains to prove the first statement for $n=\infty$.

   Suppose that $P,Q$ are two projective modules in $\mc O^{\zeta\text{-pres}+}\subseteq \mc O_\Z^+$, which exist and have finite (parabolic) Verma flags by \cite[Theorem 3.5]{CL20} and \cite[Theorem 3.12]{Le20}. That is, both $P,Q$ are direct sums of indecomposable projective covers of irreducible modules of $W_{\fl}$-antidominant highest weights. Furthermore, by \cite[Theorem 3.2]{CL20} for $n\gg 0$, we have
\begin{align}\label{eq:proj=projn}
    \text{Hom}_{\mc O_\Z^+}( P, Q)\cong \text{Hom}_{\mc O^+_{n,\Z}}(\text{Tr}_n P,\text{Tr}_n Q).
\end{align}

Note that $\Gamma_\zeta P$ and $\Gamma_\zeta Q$ are projective modules in $\mc W(\zeta)^+$ and have finite flags of standard Whittaker modules, since $\Gamma_\zeta$ is exact and sends parabolic Verma modules to the corresponding standard Whittaker modules. It follows from \eqref{eq:proj=projn} that $\text{Hom}_{W(\zeta)^+}(\Gamma_\zeta P,\Gamma_\zeta Q)<\infty$. Since $\Gamma_\zeta$ is compatible with $\text{Tr}_n$, we have for $n\gg 0$
\begin{align*}
   \text{Hom}_{\mc W(\zeta)^+}(\Gamma_\zeta P,\Gamma_\zeta Q) &\cong \text{Hom}_{\mc W(\zeta)_n^+}(\Gamma_\zeta \text{Tr}_n(P),\Gamma_\zeta \text{Tr}_n(Q))\\ &\cong \text{Hom}_{\mc O_{n}^{\zeta\text{-pres}+}}(\text{Tr}_n(P), \text{Tr}_n(Q))\\ &\cong
   \text{Hom}_{\mc O^+_{n,\Z}}( \text{Tr}_n(P), \text{Tr}_n(Q))\\&\cong
   \text{Hom}_{\mc O^+_\Z}( P, Q) \cong \text{Hom}_{\mc O^{\zeta\text{-pres}+}}( P, Q)
\end{align*}
It follows, e.g., from \cite[Proposition 5.10]{BG80}, $\mc O^{\zeta\text{-pres}+}$ is equivalent to the cokernel category with projective objects $\Gamma_\zeta P$, where $P\in \mc O^{\zeta\text{-pres}+}$. However, this category is $\mc W(\zeta)^+$. This completes the proof.
\end{proof}

Analogously, we have the following super counterpart.

\begin{prop}\label{prop:gamma+sg} Let $n\le\infty$.
The functor $\ov{\Gamma}_\zeta^+:\ov{\mc O}_{n,\Z}^+\rightarrow \ov{\mc W}(\zeta)_n^+$ induces an equivalence of categories
    \begin{align*}
        {}'{\ov{\Gamma}}_\zeta^+:\ov{\mc O}_{n,\Z}^+/\ov{\mc I}^+_\zeta\stackrel{\cong}{\longrightarrow} \ov{\mc W}(\zeta)_n^+.
    \end{align*}
    Furthermore, $\ov{\Gamma}_\zeta\left(\ov{N}(\la^\natural)\right)=\ov{N}(\la^\natural,\zeta)$, for $\la\in \Lambda_n^+$, and
   \begin{align*}
       \ov{\Gamma}_\zeta\left(\ov{L}(\la^\natural)\right)=
       \begin{cases}
           \ov{L}(\la^\natural,\zeta),\text{ if }\la\in\Lambda(\zeta)_n^+;\\
           0, \text{ otherwise}.
       \end{cases}
   \end{align*}Hence $\ov{\mc W}(\zeta)^+$ contains $\ov{N}(\la^\natural,\zeta)$ and simple Whittaker module $\ov{L}(\la^\natural,\zeta)$, for $\la\in\Lam(\zeta)_n^+$.
\end{prop}

\begin{thm} [Super duality for Whittaker modules]
\label{thm:SD:MMS}
    There exists an equivalence of properly stratified categories
    \[
    \mc W(\zeta)^+ \stackrel{\cong}{\longrightarrow} \ov{\mc W}(\zeta)^+,
    \]
    which maps $N(\zeta,\la)$ and $L(\zeta,\la)$ to $\ov{N}(\la^\natural,\zeta)$ and $\ov{L}(\la^\natural,\zeta)$, respectively, for $\la\in\Lambda(\zeta)^+$.
\end{thm}

\begin{proof}
    By Theorem \ref{thm:SD} we have an equivalence of categories between the highest weight categories $\mc O_\Z^+$ and $\ov{\mc O}_\Z^+$ under which $N(\la)$ and $L(\la)$ correspond to $\ov{N}(\la^\natural)$ and $\ov{L}(\la^\natural)$, respectively. Now the simple objects in the Serre subcategory $\mc I_\zeta^+$ correspond to the simple objects in the Serre subcategory $\ov{\mc I}^+_\zeta$. Thus, we have $\mc O_\Z^+/\mc I_\zeta^+\cong \ov{\mc O}_\Z^+/\ov{\mc I}_\zeta^+$.

    Now, by Propositions \ref{prop:gamma+g} and \ref{prop:gamma+sg} the categories $\mc W(\zeta)^+$ (and respectively, $\ov{\mc W}(\zeta)^+$) of MMS Whittaker modules are equivalent, respectively, to the Serre quotient category $\mc O_\Z^+/\mc I^+_\zeta$ (and respectively, $\ov{\mc O}_\Z^+/\ov{\mc I}^+_\zeta$). Summarizing we have the following commutative diagrams:
    \[
\xymatrix{
\mc O_\Z^+
\ar[r]
 \ar[d]^{\cong}
 & \mc O_\Z^+/\mc I_\zeta^+
 \ar[r]^{\cong}
\ar[d]^{\cong}
  & \mc W(\zeta)^+
 \ar[d]
 \\
\ov{\mc O}_\Z^+
\ar[r]
& \ov{\mc O}_\Z^+/\ov{\mc I}_\zeta^+
\ar[r]^{\cong}
& \ov{\mc W}(\zeta)^+
}
\]
Thus, the equivalence
    $\mc W(\zeta)^+\cong\ov{\mc W}(\zeta)^+$ follows.

    The correspondence of the standard and simple objects follows from the fact that the quotient functors are isomorphic to the respective Backelin functors and according to Propositions \ref{prop:gamma+g}--\ref{prop:gamma+sg} the corresponding Backelin functors send the parabolic Verma modules and simple modules in $\mc O_\Z^+$ and $\ov{\mc O}_\Z^+$ to the standard Whittaker and simple modules (if nonzero) in $\mc W(\zeta)^+$ and $\ov{\mc W}(\zeta)^+$, respectively.
\end{proof}

\begin{rem}
    In case when $\zeta=0$, Theorem \ref{thm:SD:MMS} reduces to the super duality between parabolic BGG categories modules over $\g_\infty$ and $\sg_\infty$ (see \cite{CWZ08, CW08, CL10} for type $A$ and \cite{CLW11} for type BCD; also cf.~\cite[Chapter~ 6]{CW12}). Our proof of  Theorem \ref{thm:SD:MMS} uses the super duality in this special case when $\zeta=0$, and so does not yield a new proof of this speical case.
\end{rem}

\subsection{Categorification of Fock spaces}

Let us be specific by setting $\g_n= \gl(m+n)$ and $\sg_n= \gl(m|n)$, i.e., the head diagram \framebox{$\mf H^{\mf a}$} is the Dynkin diagram of $\gl(m)$.
Denote by $[{\mc O}^+_{n,\Z}]$ and $[\ov{\mc O}^+_{n,\Z}]$ the Grothendieck groups of the categories ${\mc O}^+_{n,\Z}$ and $\ov{\mc O}^+_{n,\Z}$, respectively.
It was shown in \cite{CLW15} (as a parabolic version of Brundan-Kazhdan-Lusztig conjecture \cite{Br03}) that there exists a $U(\gl_\infty)_\Z$-module isomorphism
\begin{align}
\label{eq:VWn}
\begin{split}
[{\mc O}^+_{n,\Z}] &\cong \VV^{\otimes m} \otimes \wedge^{n}\VV,
\\
[\ov{\mc O}^+_{n,\Z}] &\cong \VV^{\otimes m} \otimes \wedge^{n}\WW,
\end{split}
\end{align}
where $\VV$ (respectively, $\WW$) is the natural representation (respectively, its restricted dual) of the integral form $U(\gl_\infty)_\Z$ of $U(\gl_\infty)$. Here and below $\wedge^{k}\VV$ and $S^k\VV$ denote the $k$th exterior and symmetric powers of $\VV$, respectively.

The isomorphisms \eqref{eq:VWn} at $n=\infty$  become
\begin{align}
\label{eq:VWinf}
\begin{split}
[{\mc O}^+_\Z] &\cong \VV^{\otimes m} \otimes \wedge^{\infty}\VV,
\\
[\ov{\mc O}^+_\Z] &\cong \VV^{\otimes m} \otimes \wedge^{\infty}\WW.
\end{split}
\end{align}
Indeed, one first constructs $q$-versions of the right-hand side \eqref{eq:VWn}--\eqref{eq:VWinf} as modules over the (Lusztig integral form of) quantum group $U(\gl_\infty)_{\Z[q,q^{-1}]}$ (cf. \cite{Lus10}) and then specialize them at $q=1$. In this case, one constructs the standard basis, canonical  and dual canonical bases on (a completion of) $\VV^{\otimes m} \otimes \wedge^{n}\VV$ and other variants of tensor spaces. Similar remarks on $q$-deformation and canonical bases apply to \eqref{eq:VWzeta} below.

On the other hand, by the results in this section we have, for $n\leq \infty$,
\begin{align}
\label{eq:VWzeta}
\begin{split}
[{\mc W}(\zeta)^+_{{n}}] &\cong S^{m_\zeta}\VV  \otimes \wedge^{{{n}}}\VV,
\\
[\ov{\mc W}(\zeta)^+_{{n}}] &\cong S^{m_\zeta}\VV \otimes \wedge^{{{n}}}\WW,
\end{split}
\end{align}
which is a parabolic variant of the categorification in \cite[Theorem~ 44]{CCM23}. Here $S^{m_\zeta}\VV =S^{m_1} \VV\otimes \ldots \otimes S^{m_r}\VV$, where $m=m_1+\ldots +m_r$ is the Jordan block type of $\zeta$ corresponding to a Levi subalgebra of $\gl(m)$. ($S^{m_\zeta}\VV$ was denoted as $\mathbb{T}^{m}_\zeta$ in \cite{CCM23}.)
One can view \eqref{eq:VWn}--\eqref{eq:VWinf} as a special case of \eqref{eq:VWzeta} for $\zeta =0$, since in this case $\ov{\mc W}(\zeta)^+$ reduces to $\ov{\mc O}^+$, $S^{m_\zeta}\VV$ reduces to $\VV^{\otimes m}$, and so on, and we are back to the setting. Indeed, one shows, as in \cite{CCM23} (generalizing \cite{CLW15}), that the standard Whittaker, tilting and simple modules in $\ov{\mc W}(\zeta)^+$ correspond to proper standard, canonical and dual canonical bases for $S^{m_\zeta}\VV \otimes \wedge^{\infty}\WW$, respectively. (There are two versions of standard bases.)

As noted in \cite{CWZ08}, there is a natural  isomorphism
\[
\wedge^{\infty}\VV \cong \wedge^{\infty}\WW,
\]
both as the basic representation of $U(\gl_\infty)_\Z$-module of level one. This induces a natural isomorphism between the two right-hand sides in \eqref{eq:VWinf} (and respectively, in \eqref{eq:VWzeta} for $n=\infty$) and an identification of standard, (dual) canonical bases between them. Building on \cite{CCM23}, such an isomorphism in the setting of \eqref{eq:VWzeta} for $n=\infty$ is a main motivation behind the super duality in Theorem~\ref{thm:SD}.

Similarly, for $\g_n$ and $\sg_n$ with head diagrams of type $BCD$, we have isomorphisms (as modules over certain $\imath$quantum group of type AIII) as in \eqref{eq:VWn}, \eqref{eq:VWinf} and \eqref{eq:VWzeta}, and all discussions above remain valid once we replace Lusztig (dual) canonical bases of type A by Bao-Wang (dual) $\imath$canonical bases of type AIII; see \cite{BW18KL, CC24} for an $\imath$canonical basis formulation of Kazhdan-Lusztig theories for BGG module category and Whittaker module category over $\sg_n$ and $\g_n$.

We conclude this section by remarking that one can construct more general parabolic variants of the categorification of \cite[Theorem 44]{CCM23} than the ones discussed above for the Fock spaces in  \eqref{eq:VWzeta}. In fact, arbitrary tensor products of symmetric and exterior powers of $\VV$ and $\WW$ can be categorified by parabolic Whittaker categories of $\gl(m|n)$-modules, with appropriately chosen simple system $\Pi$ and compatible parabolic subalgebras $\mf p$ and $\mf q$. A similar remark applies to the $\imath$quantum group versions in \cite{CC24} as well.

\section{Module categories for finite $W$-superalgebras}
\label{sec:Walgebra}

In this section, we formulate the finite $W$-superalgebras $U(\g,e)$ associated to a basic classical Lie superalgebra $\g$ and an even nilpotent element $e$ which lies in an even Levi subalgebra $\fl$ of $\g$. Then we formulate a parabolic category $\cO$ of $U(\g,e)$-modules with ``Levi subalgebra" $U(\fl,e)$.

\subsection{Finite $W$-superalgebras}
\label{subsec:finiteW}

We continue the setup in \S\ref{subsec:setup}. Recall the $\Z$-grading \eqref{eq:eigen} of a basic classical Lie superalgebra $\g$ given by $\ad\,\theta$ for an integral element $\theta\in \h$, whose degree $0$ component is an even Levi subalgebra $\fl =\g_{\theta,0}$. Recall the triangular decomposition \eqref{eq:tri} and \eqref{eq:udecomp}.

Recall the $\mf{sl}(2)$-triple $\{e,h,f\}$ in $\g$. The eigenspace decomposition of $\ad\, h$ gives rise to a Dynkin grading of $\g$:
\[
\g=\bigoplus_{j\in\Z}\g(j),
\qquad \text{ where }
\g(j) :=\{ x\in \g \mid [h,x]=jx\}.
\]

Note that $e\in \g(2)$. The element $e$ defines a linear map
\begin{align}\label{eq:chi}
 \chi:\g\longrightarrow \C,
 \qquad
 \chi(x)=(e|x).
\end{align}
This leads to an even super-skewsymmetric bilinear form
\begin{align}\label{eq:omegachi}
  \omega_\chi:\g\times\g\longrightarrow\C,
  \qquad
  \omega_\chi(x,y)=\chi([x,y]).
\end{align}
This form restricts to a non-degenerate super-symplectic bilinear form $\omega_\chi:\g(-1)\times\g(-1)\rightarrow\C$.
Note that $\dim\g(-1)_\oa$ is even. Throughout we shall assume that
\begin{align}
    \label{eq:g1even}
    \dim\g(-1)_\ob \text{ is even}.
\end{align}
(This will be automatically satisfied in the setup for super duality later on.)

One checks by \eqref{eq:t} that $\mf t$ preserves the super-symplectic form $\omega_\chi|_{\g(-1)}$. Hence we can choose a $\mf t$-invariant Lagrangian subspace $l$ of $\g(-1)$ with respect to $\omega_\chi$ and define
\begin{align}\label{eq:spacem}
    \mf m:=l\oplus\bigoplus_{j<-1}\g(j).
\end{align}
Note that $\chi$ is a character of $\mf m$ which vanishes on $\mf m_\ob$. Set
\begin{align}
  \label{eq:mchi}
\mf m_\chi:=\{m-\chi(m)\mid m\in\mf m \}
\end{align}
and let $I_\chi$ be the left ideal of $U(\g)$ generated by $\mf m_\chi$. Also let $Q_\chi$ denote the left $U(\g)$-module $U(\g)/I_\chi$. The finite $W$-superalgebra associated to $e$ is defined to be the associative superalgebra
\begin{align*}
   U(\g,e) =\text{End}_\g(Q_\chi)^{\text{opp}}.
\end{align*}
Therefore, $Q_\chi$ is a $\big(U(\g), U(\g,e)\big)$-bimodule.
It is well known (cf., e.g., \cite{Wan11}) that we can identify
\begin{align*}
    U(\g,e)&= \left(U(\g)/I_\chi\right)^{\ad\,\mf m}\\
    &=\{u+I_\chi\in U(\g)/I_\chi\mid [m,u]\in I_\chi,\forall m\in\mf m\}.
\end{align*}
Since, by choice, the Lagrangian subspace $l$ is $\mf t$-invariant, we have that $\mf m$ is $\mf t$-invariant, and thus, $\chi([t,x])=0$, for all $t\in\mf t$ and $x\in\mf m$. Therefore, we conclude that
\begin{align}  \label{eq:tUge}
    \mf t\subseteq U(\g,e).
\end{align}\

Thus, $\texttt{Sk} :=Q_\chi\otimes_{U(\g,e)} -$ defines a functor from the category of $U(\g,e)$-modules to the category of $\g$-modules on which $\mf m_\chi$ acts locally nilpotently.

\begin{prop} \cite{Zh14, SX20}
\label{prop:sk}
The functor $\texttt{Sk}$ is an equivalence of categories from the category of $U(\g,e)$-modules to the category of $\g$-modules on which $\mf m_\chi$ acts locally nilpotently.
(It remains an equivalence when restricting to subcategories of finitely generated modules.)
\end{prop}

\begin{proof}
    The equivalence $\texttt{Sk}$ for a reductive Lie algebra $\g$ was due to Skryabin (\cite[Appendix]{Pre02}). For a Lie superalgebra $\g$ with $\g(-1)$ even-dimensional (which is assumed throughout this paper), it was observed in \cite[Remark~ 3.10]{Zh14} that Skryabin's proof can be ``superized''. For Lie algebras, Losev in \cite[Theorem 1.1.4]{Los10JAMS} gave a very different proof from Skryabin's using his decomposition theorem, see \S\ref{sec:decomp:thm} below. This proof was then generalized to the superalgebra setting in \cite[Theorem 4.1]{SX20}.

Since $U(\g)$ is Noetherian, we have that finitely generated modules over $U(\g)$ are Noetherian.  From this, it follows that Skryabin's equivalence  $\texttt{Sk}$ restricts to an equivalence of the respective subcategories of finitely generated modules.
\end{proof}

\begin{rem}
Finite $W$-superalgebras are usually defined starting with a more general grading called {\em good grading} (\cite[\S 0]{EK05}) instead of a Dynkin grading, as we have done here. However, for our main purpose, the ones constructed from Dynkin gradings will be sufficient. Hence we shall restrict ourselves to this case.
\end{rem}

\subsection{Categories of $U(\g,e)$-modules} \label{sec:cat:Wmod}

Denote the centralizer of $e$ in $\g$ by
\begin{align}\label{eq:ge}
 \g_e:=\{x\in\g\mid [e,x]=0\}.
\end{align}
By \cite[Theorem 3.8]{BGK08} and its straightforward super variant, we have a (non-unique) $\mf t$-module isomorphism:
\begin{align}\label{eq:uge}
    U(\g_e)\cong U(\g,e).
\end{align}

By the $\mf t$-module isomorphism \eqref{eq:uge} we have an $\ad\, \theta$-eigenspace decomposition
\begin{align}
    U(\g,e)=U(\g,e)_0+\sum_{k\in\Z}U(\g,e)_k.
\end{align}
Set $U(\g,e)_{\ge 0}=\sum_{k\ge 0}U(\g,e)_k$ and $U(\g,e)_{>0}=\sum_{k> 0}U(\g,e)_k$. Furthermore let
\[
U(\g,e)_{\#}=U(\g,e)_{\ge 0}\cap U(\g,e)U(\g,e)_{>0},
\]
which is a two-sided ideal of $U(\g,e)_{\ge 0}$. We have by \cite[Theorem 4.1]{Los12} (or by Remark \ref{rem::wl=w0} below)
\begin{align}
  \label{eqn:iso:1}
    U(\g,e)_{\ge 0}/U(\g,e)_{\#}\cong U(\mf l,e).
\end{align}
Denote by $\widetilde{\mc O}(\theta,e)$ the category of finitely generated $U(\g,e)$-modules $M$ such that for any $x\in M$ there exists $k_x\in\Z$ with $U(\g,e)_k x=0$ for all $k\ge k_x$.

Now, given a $U(\mf l,e)$-module $V$, we can regard it as a $U(\g,e)_{\ge 0}$-module via the isomorphism \eqref{eqn:iso:1}. Thus, we can define the induced module
\begin{align}\label{eq::ind:w}
{\mc M}^{\theta,e}(V)=U(\g,e)\otimes_{U(\g,e)_{\ge 0}}V.
\end{align}
This gives rise to a functor $\mc M^{\theta,e}$ from the category of finitely generated $U(\mf l,e)$-modules to $\widetilde{\mc O}(\theta,e)$.

On the other hand, for $M\in\widetilde{\mc O}(\theta,e)$ we let
\[
\mc F(M) =\{x\in M\mid ux=0,\forall u\in U(\g,e)_{>0}\}
\]
so that $\mc F(M)$ is naturally a $U(\mf l,e)$-module by \eqref{eqn:iso:1} again. The functor $\mc F$ is right adjoint to $\mc M^{\theta,e}$.

Let $\mc O(\theta,e)$ be the subcategory of $\widetilde{\mc O}(\theta,e)$ of $U(\g,e)$-modules $M$ for which $\dim\mc F(M)<\infty$.  The functor $\mc M^{\theta,e}$ restricts to a functor from the category of finite-dimensional $U(\mf l,e)$-modules to $\mc O(\theta,e)$ with right adjoint being the restriction of the functor $\mc F$ above:
\[
\xymatrix{
U(\mf l,e)\mod
 \ar[r]^{\quad \mc M^{\theta,e}}
 & \mc O(\theta,e)  ,
& U(\mf l,e)\mod
& \mc O(\theta,e)
\ar[l]_{\quad \mc F} .
}
\]
Also, we let $\mc O^{\mf r}(\theta,e)$ be the full subcategory of $\mc O(\theta,e)$ consisting of objects on which $\mf r$ acts semisimply.

Now if $V$ is a finite-dimensional irreducible $U(\mf l,e)$-module, then $\mc M^{\theta,e}(V)$ has a composition series and a unique irreducible quotient, denoted by $\mc L^{\theta,e}(V)$ according to \cite[Corollary 4.11]{BGK08} (see also \cite[Corollary~ 3.6, Proposition 3.7]{Los12}). Note that if $V$ is finite dimensional and semisimple over $\mf r$, then the module $\mc M^{\theta,e}(V)$ lies in $\mc O^{\mf r}(\theta,e)$.

Suppose that we have two integral elements $\theta,\theta'\in\mf t$ that give rise to two minimal full subalgebras $\mf r,\mf r'$ such that $\mf r\subseteq\mf r'\subseteq \mf t$. Since the parabolic subalgebras \eqref{eq:para} determined by $\theta$ and $\theta'$ are assumed to be compatible with the same triangular decomposition \eqref{eq:tri} we conclude that $\mc O(\theta,e)\subseteq\mc O(\theta',e)$ and $\mc O^{\mf r}(\theta,e)\subseteq\mc O^{\mf r'}(\theta',e)$.

\begin{ex}
When $\mf r=\mf t$ the category $\mc O^{\mf t}(\theta,e)$ is the original category $\mc O$ for finite $W$-algebra defined in \cite{BK08}. In this case $e$ is principal nilpotent in the Levi subalgebra $\mf l$, and hence the irreducible $U(\mf l,e)$-modules are one-dimensional according to \cite{Kos78}. If $V$ is a one-dimensional $U(\mf l,e)$-module, then $\mc M^{\theta,e}(V)\in\mc O^{\mf t}(\theta,e)$ is the Verma module of \cite{BK08}.
\end{ex}

\subsection{An example}
Let $\g=\gl(7)$ and let $e$ be the nilpotent element associated with the following pyramid, see, e.g., \cite[\S3.1]{BK08}.
\begin{align}
   \begin{ytableau}
    \none & 1 & \none \\
    2 & 3  & 4\\
    5 & 6  & 7
\end{ytableau}
\end{align}
The corresponding nilpotent element is $e=E_{23}+E_{34}+E_{56}+E_{67}$ with the grading operator $h=2E_{22}+2E_{55}-2E_{44}-2E_{77}$. Together with the element $f=2\left(E_{32}+E_{43}+E_{65}+E_{76}\right)$ they form an $\mf{sl}(2)$-triple inside $\g=\mf{gl}(7)$. Here we have used $E_{ij}$ to denote the usual elementary matrix of $\mf{gl}(n)$. The eigenvalues of ad$h$ give rise to the Dynkin grading $\g=\oplus_{j\in\Z}\g(j)$. Note that $\g(0)$ is the subalgebra corresponding to the columns, i.e., $\g(0)\cong\gl(2)\oplus\gl(3)\oplus\gl(2)$, generated by the root vectors corresponding to the roots $$\{\pm(\ep_2-\ep_5),\pm(\ep_1-\ep_3),\pm(\ep_3-\ep_6),\pm(\ep_4-\ep_7)\}.$$
 This grading determines a $W$-algebra $U(\g,e)$.
 The subalgebra $\mf t$ is $3$-dimensional and spanned by the following basis $$\{E_{11}, E_{22}+E_{33}+E_{44}, E_{55}+E_{66}+E_{77}\}.$$ We have a $\mf t$-module isomorphism between $U(\g,e)$ and $U(\g_e)$.
 An ``integral element'' $\theta\in\mf t$ is of the form
\begin{align*}
   \theta=\theta_1 E_{11}+\theta_2\left(E_{22}+E_{33}+E_{44}\right)+\theta_3\left(E_{55}+E_{66}+E_{77}\right),
\end{align*}
where $\theta_i\in\Z$. The eigenvalues of $\theta$ gives rise to a $\Z$-gradation $\g=\oplus_{k\in\Z}\g_k$ with $\fl =\g_0$.

In the case when all $\theta_i$ are the same, $\theta$ is a multiple of the identity element $I$, and hence it gives the eigenvalue decomposition $\g=\g_0=\mf l$ and hence $\g(0)\cap\g_0=\g(0)$. The ``Levi subalgebra" $U(\mf l,e)$ of $U(\g,e)$ is $U(\g,e)$ itself.

Now consider the case when all $\theta_i$ are distinct, i.e., $\theta$ is regular (e.g., $\theta_1=1,\theta_2=0,\theta_3=-1$). In this case $\mf l=\g_0$ is the subalgebra corresponding to the rows, i.e., $\mf l\cong\gl(1)\oplus\gl(3)\oplus\gl(3)$, generated by the Cartan subalgebra and the root vectors corresponding to the roots $$\{\pm(\ep_2-\ep_3),\pm(\ep_3-\ep_4),\pm(\ep_5-\ep_6),\pm(\ep_6-\ep_7)\}.$$
The ``Levi subalgebra" $U(\mf l,e)$ of $U(\g,e)$ in this case is commutative and isomorphic to the center of $U(\mf l)$.

\section{Losev-Shu-Xiao decomposition and finite $W$-superalgebras}\label{sec:Wsuper duality}

In this section, we set up (in a slightly extended form which we need) Shu-Xiao's super generalization of Losev's approach to finite $W$-algebras. As consequences we formulate equivalences between categories of Whittaker $\g$-modules and module categories for finite $W$-superalgebras. Then we formulate the super duality between a finite $W$-algebra and a finite $W$-superalgebra at infinite-rank limit.

\subsection{Super Darboux-Weinstein decomposition}

Recall the $\mf{sl}(2)$-triple $\{e,h,f\}$ in a basic classical Lie superalgebra $\g$ as in \S\ref{subsec:setup}, the linear map $\chi:\g\rightarrow\C$ from \eqref{eq:chi} given by $(e|-)$. We have a vector space decomposition: $\g=\g_e\oplus [f,\g]$, and the super-skewsymmetric bilinear form $\omega_\chi$ from \eqref{eq:omegachi} restricted on $[f,\g]$ is non-degenerate. Recall our assumption \eqref{eq:g1even} that $\dim_\C\g(-1)$ is even. Let $\mf m$ be as in \eqref{eq:spacem}, which is a Langragian subspace of $[f,\g]$ lying in the negative degree component of $\g$, and let $\mf m^* \subseteq [f, \g]$ be a ``dual'' of $\mf m$ with respect to $\omega_\chi$.  Let $V:=\mf m_\chi\oplus \mf m^*$. We note that $h$ acts on both $V$ and $\g_e$ via the Kazhdan grading, which is the grading on $\g$ determined by the eigenvalues of $\ad\,h$ shifted by $2$.

Write $\g_\chi=\{g-\chi(g)|g\in\g\}$ so that we have:
\begin{align*}
    \g_\chi=\g_e\oplus V.
\end{align*}
Note that $S(\g)$ has a Poisson structure given by the Lie super bracket. On the other hand, since $V$ is a symplectic superspace it has a Poisson structure given by the super-symplectic form.

The adjoint action of $\mf t\times \C h$ on $S(\g)$ integrates to an action of the adjoint group $T\times \C^\times$, where  $\C^\times$ corresponds to a subgroup in the adjoint group of $\g_\oa$ that gives rise to the Kazhdan grading.

The following is an equivariant version of super Darboux-Weinstein theorem \cite[Theorem 1.3]{SX20}. The version stated and proved in loc.~cit.~is a $\C^\times$-equivariant version, but the same proof works in the $T\times \C^\times$-equivariant version below (cf.~the proof of \cite[Theorem 3.3.1]{Los10JAMS}).

\begin{prop}\label{prop::DW}
    We have a $T\times \C^\times$-equivariant isomorphism of Poisson algebras
\begin{align*}
    S(\g)^{\wedge}_\chi\cong S(\g_e)^{\wedge}_\chi\widehat{\otimes} S(V)^{\wedge}_0.
\end{align*}
The notation ${\cdot}^{\wedge}_\chi$ denotes completion with respect to the maximal ideal corresponding to the point $\chi$. Furthermore, $S(\g_e)$ and $S(V)$ Poisson commute with each other.
\end{prop}

\begin{rem}
Though not explicitly addressed in \cite{SX20}, the notation $\widehat{\otimes}$ in Proposition~ \ref{prop::DW} (and similarly in Theorem~ \ref{thm:qDW} below) means taking the completion of the tensor product with respect to the maximal ideal corresponding to the point $(\chi,0)\in \g_e^*\times V^*$ (see \cite[Proposition 2.1]{Los12}). Here, as a vector superspace, this ``completed'' tensor product is isomorphic to $S(\g_e\times V)^\wedge_{(\chi,0)}$. Note also that $S(V)^{\wedge}_0$ is just the formal power series in the variables of $V$.
\end{rem}

\subsection{A super setting for Losev's decomposition}\label{sec:decomp:thm}

For applications to finite $W$-(super)algebras we need a quantum version of Proposition \ref{prop::DW}. The quantum analogue of a Poisson structure is given by a star product.

Given a Poisson superalgebra $(A,\{\cdot,\cdot\})$ over $\C$, a star product on $A$ is an associative product $*$ on $A[[\hbar]] = A \otimes_\C \C[[\hbar]]$ of the form:
\begin{align*}
    f*g=fg+\sum_{i=1}^\infty D_i(f,g)\hbar^{2i}
\end{align*}
where $\hbar$ is a formal parameter and $D_i(f,g)\in A$. Furthermore, it needs to satisfy among others the following condition:
\begin{align*}
    f*g- (-1)^{|f|\cdot |g|} g*f-\{f,g\}\in \hbar^3 A[[\hbar]],
\end{align*}
where $f,g \in A$ are assumed to be $\Z_2$-homogeneous of degree $|f|, |g|$.
We will not list the other so-called continuity conditions which can be found in \cite[Section 2]{Los12}. They are in place to assure that the star product is homogeneous with respect to the Kazhdan grading and that it can be extended to various completions, and also makes sense when we set $\hbar=1$.

\begin{ex} \label{ex:star:g} Let $\g$ be a Lie superalgebra.
  As the standard quantization of the Poisson superalgebra $(S(\g),[\cdot,\cdot])$, we have $(S(\g)[[\hbar]],*) :=T(\g)[[\hbar]]/I$, where $I$ is the ideal generated by   $a\otimes b- (-1)^{|a|\cdot |b|} b\otimes a -[a,b]\hbar^2$, for $a,b \in \g$ $\Z_2$-homogeneous. By a PBW-type argument we see that the $T(\g)[[\hbar]]/I$ is an associative superalgebra isomorphic to $S(\g)[[\hbar]]$ as a vector space, and hence gives rise to a star product on $S(\g)$, called the Gutt star product.
\end{ex}

\begin{ex}\label{ex:star:v}
 Suppose that $(V,\omega)$ is a symplectic superspace so that $S(V)$ is naturally a Poisson superalgebra with Poisson bracket determined by $\{v,w\}:=\omega(v,w)$, $v,w\in V$. It has a standard quantizaton $(S(V)[[\hbar]],*)$ satisfying the relation $v*w- (-1)^{|v|\cdot |w|} w*v=\omega(v,w)\hbar^2$. The corresponding star product is a Moyal-Weyl star product associated with a constant nondegenerate bivector on $V$.
\end{ex}

Recall the action of the group $T\times \C^\times$ on $S(\g)$, where $\C^\times$ is a subgroup in the adjoint group of $\g_\oa$ determining the Kazhdan grading. Letting $t\cdot\hbar=t\hbar$ for $t\in\C^\times$, we see that star product of $(S(\g)[[\hbar]],*)$ is homogeneous and $T\times \C^\times$-equivariant. Furthermore, as mentioned above, the ``continuity'' conditions of the star product allows us to extend the quantum algebra structures in Examples \ref{ex:star:g}--\ref{ex:star:v} to their respective completions as defined in Proposition \ref{prop::DW}.

We are now ready to state a slight upgrade of \cite[Theorem 1.6]{SX20}, which was formulated as a $\C^\times$-equivariant version only. However, the proof in loc.~cit.~extends to the $T\times \C^\times$-version that we shall need for our application to the category $\mc O$ of $W$-superalgebras. We note that in the case when $\g$ is a reductive Lie algebra an even stronger equivariant version of the theorem was proved in \cite[Proposition 2.1]{Los12} derived from the earlier results in \cite{Los10JAMS}.

\begin{thm} \rm{(cf. \cite[Theorem 1.6]{SX20})}
\label{thm:qDW}
    We have a $T\times \C^\times$-equivariant isomorphism of the quantum algebras
\begin{align}
 \label{isoSSS}
    S(\g)^{\wedge}_\chi[[\hbar]]\cong S(\g_e)^{\wedge}_\chi[[\hbar]]\widehat{\otimes}_{\C[[\hbar]]} S(V)^{\wedge}_0[[\hbar]].
\end{align}
\end{thm}

\begin{ex}\label{ex:even:skryabin}
The isomorphism in Theorem \ref{thm:qDW} is equivariant with respect to the Kazhdan grading. Now, suppose that the Dynkin grading on $\g$ is even (i.e., $\g=\bigoplus_{j\in 2\Z}\g(j)$) so that the subalgebra $\mf m_\chi$ consists precisely of the non-positively graded components of $\g$ in the Kazhdan grading. The isomorphism in Theorem \ref{thm:qDW} restricts to an isomorphism of the corresponding $\C^\times$-finite part (with respect to the Kazhdan grading). Setting $\hbar=1$, we get the following isomorphism of associative superalgebras: The left-hand side is precisely $U(\g)^{\wedge}_{\mf m_\chi}$, while the right-hand side is the tensor product of the Weyl algebra of $\mf m_\chi\oplus\mf m^*$ (completed with respect to $\mf m_\chi$) and $S(\g_e)[[\hbar]]/(\hbar-1)$, since $\g_e$ is positively graded with respect to the Kazhdan grading. From the representation theory of Weyl algebra, it follows that the category of $\g$-modules on which $\mf m_\chi$ acts locally nilpotently is equivalent to the $S(\g_e)[[\hbar]]/(\hbar-1)$-module category. Now, we observe that
$$
U(\g,e) =\left(U(\g)/U(\g)\mf m_\chi\right)^{\ad\,\mf m}\cong \big(U(\g)^{\wedge}_{\mf m_\chi}/U(\g)^{\wedge}_{\mf m_\chi}\mf m_\chi\big)^{\ad\,\mf m}\cong S(\g_e)[[\hbar]]/(\hbar-1).
$$
Therefore, the $S(\g_e)[[\hbar]]/(\hbar-1)$-module category above is just the category of $U(\g,e)$-module. Thus, in the case when the Dynkin grading is even, Theorem~ \ref{thm:qDW} readily implies the Skryabin equvalence in Proposition~ \ref{prop:sk}.
\end{ex}

For the case when the Dynkin grading is not even, and especially in order to make connection between categories of Whittaker modules over Lie superalgebras of Section \ref{sec:SD:MMS} and categories modules over finite $W$-(super)algebras, we need to extend the isomorphism in Theorem \ref{thm:qDW} between the $\C^\times$-finite parts which we shall explain below.

Let $\theta\in\mf t$ be an integral element giving rise to a parabolic decomposition of $\g=\mf u^-\oplus\mf l\oplus\mf u$ as in \eqref{eq:paracomp}. Let $d$ be the maximal eigenvalue of $\ad h$. Let $m>2d+2$, and consider the element $h-m\theta\in \mf t\oplus\C h$. The eigenvalue of this element gives rise to a $\Z$-gradation on $\g_\chi=\oplus_{i\in\Z}\g_\chi(i)$, where we recal that $h$ acts by the Kazhdan grading. That is
\begin{align}\label{new:grading}
    \g_\chi(i):=\{x\in\g_{\chi} \mid [h-m\theta,x]=(i-2)x\}.
\end{align}
Let
\begin{align}
    \label{eq:mchi2}
    \widetilde{\mf m}:=\underline{\mf m}_\chi+\mf u,
\end{align}
where $\underline{\mf m}_\chi$ is the Levi $\mf l$ analogue of $\mf m_\chi$ in \eqref{eq:mchi} for $\g$. Observe that, since $\theta\in \mf t$, the subspace $\widetilde{\mf m}\cap V$ is Langragian in $V$. Again, if $\theta=0$, this is just the Kazhdan grading, and $\widetilde{\mf m}=\mf m_\chi$.

Let $\C^\times\rightarrow T\times \C^\times$ be the diagonal embedding such that the pull-back $\C^\times$-action induces the grading \eqref{new:grading}, i.e., the differential of the embedding at $1$ equals $(h,-m\theta)$.  The following theorem follows from \cite[Lemma 3.7]{SX20}, adapted to our new $\C^\times$-equivariant setting and \cite[Proposition 5.1]{Los12}. We observe that the theorem below in the special case $\theta=0$ is precisely \cite[Theorem 1.7]{SX20}.

\begin{thm}
\label{thm:catOW}
The new $\C^\times$-equivariant isomorphism induced from the one in \eqref{isoSSS} restricts to an isomorphism of the corresponding $\C^\times$-finite parts. This $\C^\times$-finite part isomorphism in turn extends uniquely to an isomorphism of algebras
    \begin{align}
      \label{isoUUH}
    U(\g)_{\widetilde{\mf m}}^{\wedge}\cong  U(\g,e)_{\widetilde{\mf m}\cap \g_e}^{\wedge} \otimes {\bf A}(V)_{\widetilde{\mf m}\cap V}^{\wedge},
    \end{align}
    where ${\bf A}(V)$ is the Weyl superalgebra of the supersympectic space $V$.
\end{thm}

\begin{rem}\label{rem::wl=w0}
Following \cite[(5.6)]{Los12} we can take the subalgebras of non-negative $\ad\,\theta$-eigenspaces divided by their left ideals generated by positive $\ad\,\theta$-eigenspaces in both sides of the isomorphism \eqref{isoUUH}. Since the subspace consisting of $\ad\,\theta$-positive eigenspaces is a two-sided ideal in the algebra of non-negative $\ad\,\theta$-eigenspaces, we see that these left ideals are indeed two-sided ideals, and hence the quotients are indeed algebras. Thus, we obtain an isomorphism of algebras
    \begin{align*}
        U(\mf l)_{\underline{\mf m}}^{\wedge}\cong  U(\g,e)_{\ge 0}/U(\g,e)_{\#}  \otimes {\bf A}(V\cap\mf l)_{\underline{\mf m}}^{\wedge}.
    \end{align*}
    Now dividing by the left ideal generated by $\underline{\mf m}_\chi$ and then taking $\underline{\mf m}$-invariants on both sides, we obtain the isomorphism \eqref{eqn:iso:1}.
\end{rem}

\subsection{Categories of generalized Whittaker modules}\label{subsec:genWH}

We continue the setup in \S\ref{subsec:setup} with  $\h \subset \fl \subset \g$ and an $\mf{sl}(2)$-triple $\{e,h,f\}$. Since $h\in \h$ and $\h \subset \fl$ by \eqref{eq:udecomp}, $\ad\, h$ preserves $\fl$. Thus $\fl$ inherits a $\Z$-gradation from the Dynkin grading (with respect to $\ad\,h$) $\g=\oplus_{j\in \Z}\g(j)$:
\begin{align*}
    \mf l=\oplus_{j\in\Z} \mf l(j),\qquad \text{ where } \mf l(j):=\mf l\cap \g(j).
\end{align*}
Recall from \eqref{eq:el} that $e\in\mf l$. From now on, we further assume that
\begin{align}
\label{eq:eLevi}
    e \text{ is of standard Levi type},
\end{align}
that is, it is principal nilpotent in a Levi subalgebra of $\mf l$ (or of $\g$).

As in \S\ref{subsec:finiteW} (with $\fl$ in place of $\g$), we have the finite $W$-algebra $U(\mf l,e)$ associated to the nilpotent element $e\in\mf l$. By Skryabin equivalence in Proposition~ \ref{prop:sk}, the category of $U(\mf l,e)$-modules is equivalent to the category of $U(\mf l)$-modules on which a corresponding subalgebra $\underline{\mf m}$ in $\mf l$ transforms by the character $\chi$, i.e., the algebra $\underline{\mf m}_\chi=\{x-\chi(x)|x\in\underline{\mf m}\}$ acts locally nilpotently.

\begin{definition}
A finitely generated $U(\g)$-module is called a generalized Whittaker module corresponding to the pair $(\theta, e)$, if the subalgebra $\underline{\mf m}_\chi+\mf u$ acts locally nilpotently.
\end{definition}
Denote the category of generalized Whittaker $\g$-modules by $\widetilde{\W}(\theta,e)$, cf. \cite[Section 4]{Los12}.

Given a finite-dimensional $U(\mf l,e)$-module $V$, we have by Proposition~ \ref{prop:sk} a $U(\mf l)$-module $\texttt{Sk}(V)$, which we can extend trivially to a $U(\mf p)$-modules and then parabolically induce to a $U(\g)$-module
\begin{align*}
    M^{\theta,e}(V):=\text{Ind}_{\mf p}^\g \texttt{Sk}(V).
\end{align*}
Then $M^{\theta,e}(V)$ is a generalized Whittaker module associated with $(\theta,e)$. In this way, $M^{\theta,e}$ defines a functor from finitely generated $U(\mf l,e)$-modules to $\widetilde{\W}(\theta,e)$.
On the other hand, given a generalized Whittaker $\g$-module $M$ we can define $F(M):=\{m\in M\mid xm=0,\forall x\in\underline{\mf m}_\chi+\mf u\}$. Then $F(M)$ is a $U(\mf l,e)$-module. The functor $F$ is right adjoint to $M^{\theta,e}$.

The subcategory $\W(\theta,e)$ of $\widetilde{\W}(\theta,e)$ is the category of generalized Whittaker modules $M$ associated with $(\theta, e)$ such that $\dim F(M)<\infty$. Furthermore, we define $\W^{\mf r}(\theta,e)$ to be the full subcategory of $\W(\theta,e)$ on which $\mf r$ acts semisimply. In the superalgebra setting, it is convenient to consider the full subcategory $\W'(\theta,e)$ of $\W(\theta,e)$ consisting of objects $M$ such that $\dim M^{\mf u_{\bar 0}+\underline{\mf m}_\chi}<\infty$; see Lemma~\ref{lem:MMS=N} below in the case when $\theta$ is regular in $\mf t$ (i.e., $e$ is a principal nilpotent element in the Levi subalgebra $\mf l$).

In the case when $\theta$ is regular in $\mf t$, we shall sometimes drop $\theta$, i.e., we write $M^e(V)$, $\widetilde{\W}(e)$, $\W(e)$, $\W'(e)$ etc.

\begin{lem} \label{lem:MMS=N}
Suppose that $e$ is principal nilpotent in $\mf l$, i.e., $\theta\in\mf t$ is regular.
\begin{enumerate}
    \item
Let $M\in\widetilde{\W}(e)$. Then $\left(M^{\mf u_{\bar 0}}\right)^{\underline{\mf m}_\chi}$ is finite dimensional if and only if  $M\in\MS(\zeta)$.
\item We have an equivalence of categories $\MS(\zeta)\cong\W'(e)$.
\end{enumerate}
\end{lem}

\begin{proof}
    When $\g$ is a Lie algebra it is shown in \cite[Lemma 4.2]{Los12} that $M\in\MS(\zeta)$ if and only if $\dim F(M)<\infty$.

Suppose that $\g$ is a basic classical Lie superalgebra and $M\in{\W}'(e)$. By definition $\left(M^{\mf u_\oa}\right)^{\underline{\mf m}_\chi}$ is finite dimensional. Thus, by \cite[Lemma 4.2]{Los12} the Noetherian $\g_\oa$-module $\text{Res}^\g_{\g_{\bar 0}} M$ lies in the MMS category of Whittaker modules of $\g_\oa$, and hence is $Z(\g_\oa)$-finite. Therefore, it lies in $\MS(\zeta)$.

Conversely, if $M\in\MS(\zeta)$, then, by definition, the Noetherian $\g_\oa$-module $\text{Res}^\g_{\g_\oa} M$ is $Z(\g_\oa)$-finite, and hence by \cite[Lemma 4.2]{Los12} again, $\left(M^{\mf u_{\bar 0}}\right)^{\underline{\mf m}_\chi}$ is finite dimensional. Hence $M\in \W'(e)$.
\end{proof}

\subsection{Several category equivalences}

Now we shall interpret the Losev-Shu-Xiao decomposition in Theorem \ref{thm:catOW} from the viewpoints of Whittaker modules and $U(\g,e)$-modules.

Recall from parabolic decomposition determined by an integral element $\theta$: $\g=\mf u^-\oplus\mf l\oplus\mf u$.  Recall the subalgebra $\widetilde{\mf m}=\underline{\mf m}_\chi+\mf u$ from Theorem \ref{thm:catOW}. So the category of $U(\g)^{\wedge}_{\widetilde{\mf m}}$-modules above is precisely $\widetilde{\mf{Wh}}(\theta,e)$.

On the other hand, we have $\widetilde{\mf m}\cap \g_e=\mf u\cap\g_e$. Also, recall that for $\alpha\in\Phi$ we have $\alpha(\theta)>0$ if and only if $\alpha$ is a root in $\mf u$. Therefore, the category of finitely generated $U(\g,e)_{\widetilde{\mf m}\cap\g_e}$-modules consists precisely of $U(\g,e)$-modules on which $U(\g,e)_{>0}$ acts locally nilpotently, i.e., the category $\widetilde{\mc O}(\theta,e)$ in \S\ref{sec:cat:Wmod}.

Finally, we have already observed earlier that the space $\widetilde{\mf m}\cap V$ is a Lagrangian subspace of $V$. Hence ${\bf A}(V)$ is just the Weyl superalgebra of $(\widetilde{\mf m}\cap V) \oplus (\widetilde{\mf m}\cap V)^*$. Thus, the category of ${\bf A}(V)_{\widetilde{\mf m}\cap V}$-modules is semisimple with a unique simple object.

Thus we have obtained a functor
\begin{align}   \label{eq:equivK}
\mc K: \widetilde{\W}(\theta,e)
\longrightarrow \widetilde{\mc O}(\theta,e),
\end{align}
which is an equivalence.

Recall the subcategory ${\mc O}(\theta,e) \subset \widetilde{\mc O}(\theta,e)$ from \S\ref{sec:cat:Wmod} and the subcategory $\W(\theta,e) \subset \widetilde{\mf{Wh}}(\theta,e)$ from \S\ref{subsec:genWH}. By restriction, we obtain a category equivalence
\begin{align}   \label{eq:equivK2}
\W(\theta,e) \stackrel{\cong}{\longrightarrow} {\mc O}(\theta,e).
\end{align}

Now consider the case when $e$ is principal nilpotent in $\mf l$. Up to conjugation we can assume that $e=\sum_{\alpha\in\Pi_\fl}E_{-\alpha}$, where the simple root vectors are normalized so that $(E_{-\alpha}|E_\beta)=\delta_{\alpha,\beta}$, for even simple roots $\alpha, \beta$. Since $e$ is principal nilpotent in $\mf l$, we see that the algebra $\widetilde{\mf m}=\underline{\mf m}+\mf u$ in this case is a maximal nilpotent subalgebra of $\g$, and so $\mf b:=\mf h+\widetilde{\mf m}$ is a Borel subalgebra of $\g$ with simple system $\Pi$. Note that for $\alpha\in\Pi$, we have that  $\zeta(E_\alpha):=(e|E_\alpha)=1$, if and only $\alpha \in \Pi_\fl$. This determines a unique character $\zeta$ of the nilradical $[\mf b,\mf b]\subseteq\mf b$ and hence the category of $U(\g)^{\wedge}_{\widetilde{\mf m}}$ in this case is the category $\g$-modules on which $x-\zeta(x)$ acts locally nilpotently for all $x\in[\fb,\fb]$. In particular, if we further restrict to the subcategory of $Z(\g_\oa)$-finite modules, then we obtain precisely the category of Whittaker modules $\MS(\zeta)$, considered in \cite{McD85, MS97} in the case when $\g$ is a Lie algebra, and in \cite{Ch21} in the case when $\g$ is a Lie superalgebra.

We define $\mc O'(\theta,e)$ to be the full subcategory of $\mc O(\theta,e)$ which is the image category of $\MS(\zeta)$ under the equivalence $\mc K$. Summarizing we have obtained the following.

\begin{thm}
\label{thm:isoK}
    We have the following commutative diagram:
\begin{align}
  \label{equiv:mcK}
\xymatrix{
\MS(\zeta)
\ar[d]^{\cong}
 \ar@{^{(}->}[r]
& \W(\theta,e)
\ar[d]^{\cong}
 \ar@{^{(}->}[r]
 & \widetilde{\W}(\theta,e)
\ar[d]_{\mc K}^{\cong}
 \\
\mc O'(\theta,e)
 \ar@{^{(}->}[r]
 & {\mc O}(\theta,e)
  \ar@{^{(}->}[r]
& \widetilde{\mc O}(\theta,e)
}
\end{align}
\end{thm}

\begin{proof}
    The equivalence $\mc K$ follows from the isomorphism of superalgebras in Theorem~ \ref{thm:catOW} and the discussion above. The remaining two equivalences follow by restriction.
\end{proof}

\begin{rem}
When $\g$ is a reductive Lie algebra, the equivalence \eqref{eq:equivK2} was first conjectured in \cite[Conjecture 5.3]{BGK08}, and then established, along with \eqref{eq:equivK}, in \cite[Theorem 4.1]{Los12}. In this case, $\MS(\zeta)\cong\W(\theta,e)$; see Lemma~\ref{lem:MMS=N}.
\end{rem}

\begin{rem}
Setting $\theta=0$, $\mc K$ gives Skryabin equivalence for finite $W$-superalgebras associated with a general not necessarily even grading \cite[Theorem 4.1]{SX20}, which extends the case considered in Example \ref{ex:even:skryabin}.
\end{rem}

\begin{cor}\label{cor::verma} Assume that $\theta$ is regular so that $e$ is principal nilpotent in $\mf l$.
The $U(\g,e)$-Verma module lies in $\mc O'(\theta,e)$, and so does its irreducible quotient.
\end{cor}

\begin{proof}
    For $M\in\widetilde{\W}(\theta,e)$, we have the following isomorphism of $U(\mf l,e)$-modules by Theorem \ref{thm:catOW}:
    \begin{align*}
        \mc F\left(\mc K(M)\right)\cong \mc G(M).
    \end{align*}
    Thus, these two functors are isomorphic, and hence so are their left adjoints $\mc K\circ M^{\theta,e}$ and $\mc M^{\theta,e}$. Taking $N$ to be the irreducible one-dimensional $U(\mf l,e)$-module, we conclude that
    \begin{align*}
        \mc K\big(M^{\theta,e}(N)\big)\cong \mc M^{\theta,e}(N).
    \end{align*}
    That is, the standard Whittaker modules of $\g$ under $\mc K$ correspond to the Verma modules of $U(\g,e)$. But we have seen earlier that the abelian category $\MS(\zeta)$ contains all standard Whittaker modules of $\g$.
\end{proof}

\begin{rem}
    \label{rem:GG}
    The Grothendieck groups of $\MS(\zeta)$ and $\W(\theta,e)$ coincide as these two categories have the same simple objects. It follows by Theorem~\ref{thm:isoK} that the Grothendieck groups of $\mc O'(\theta,e)$ and ${\mc O}(\theta,e)$ coincide as well.
\end{rem}

\begin{rem}\label{rem:comp:verma}
Let $\g$ be a basic classical Lie superalgebra and $e$ an even nilpotent element that is principal nilpotent in an even Levi subalgebra of $\g$. Suppose that the corresponding integral element $\theta$ is regular so that we have the corresponding category $\mc O(e)$ of $U(\g,e)$-modules. Let $\zeta$ be the nilcharacter corresponding to $e$, and consider the corresponding category $\mc{MS}(\zeta)$ of MMS Whittaker modules so that we have, as in Theorem~ \ref{thm:isoK}, $\mc{MS}(\zeta)\cong\mc O'(e)\subset\mc O(e)$. Now by \cite[Theorem~ 1]{CC24} (which extends \cite[Theorem~ 20]{Ch21}), the multiplicities of composition factors of standard Whittaker modules are given by the corresponding super Kazhdan-Luszig polynomials for the BGG category $\mc O$ of $\g$-modules.  Together with Theorem~ \ref{thm:isoK}, it now follows that these same Kazhdan-Lusztig polynomials also compute the multiplicities of composition factors of the Verma modules in $\mc O(e)$.  When $\g$ is a reductive Lie algebra such a relationship was first conjectured in \cite[Conjecture~ 5.3]{BGK08} and proved in \cite[Theorem~ 4.1]{Los10}. So, combining Theorem~ \ref{thm:isoK} with \cite[Theorem~ 1]{CC24} indeed establishes the analogue of the Brundan-Goodwin-Kleshchev-Losev result for basic classical Lie superalgebras.
\end{rem}

Let $\theta$ and $\theta'$ be compatible as in \S\ref{sec:cat:Wmod}. Assume that $\theta$ is regular so that $e$ is principal nilpotent in the Levi subalgebra $\mf l$ in the parabolic subalgebra determined by $\theta$. Then the abelian subcategory $\mc O'(\theta',e):=\mc O(\theta',e)\cap\mc O'(\theta,e)$ of $\mc O'(\theta,e)$ contains all Verma modules in $\mc O(\theta',e)$, since they are quotient of the Verma modules of $\mc O(\theta,e)$, which by Corollary \ref{cor::verma} are all contained in $\mc O'(\theta,e)$. Hence all simple objects in $\mc O(\theta',e)$ are also contained in $\mc O'(\theta',e)$. Under the equivalence of categories the corresponding category of generalized Whittaker modules, which we denote by $\W'(\theta',e)$, is a full subcategory of $\MS (\zeta)$. Namely, the category $\W'(\theta',e)$ is the full subcategory of $\MS(\zeta)$ consisting of modules $M$ with $\dim M^{\mf u'+\underline{\mf m}_\chi'}<\infty$, where $\mf u'$ and $\underline{\mf m}'_\chi$ are the respective subalgebras defined for the integral element $\theta'$. In particular when $\theta'=0$ so that $\mf u'=0$, the condition becomes $\dim M^{\mf m_\chi}<\infty$, i.e., the Whittaker modules in $\MS(\zeta)$ which, under the Skryabin equivalence, correspond to finite-dimensional modules of $U(\g,e)$.

\subsection{Super duality for finite $W$-algebras}

In this last subsection, we are back to the setting of \S\ref{sec:SD1}, where we introduced a pair of Lie algebra $\g_n$ and Lie superalgebra $\sg_n$, for $n\ge 1$.

Now, we can combine our results to establish a super duality between categories of modules over the finite $W$-algebra $U(\g_n,e)$ and $U(\sg_n,e)$ in the limit $n\to\infty$. Recall that $\mf{H}$ and $\ufk\cong\gl(n)$ denote the Lie algebras corresponding to the head and tail diagrams, respectively, of the Dynkin diagrams of $\g_n$ and $\sg_n$ that form a super duality pair when $n\to\infty$. We have that $e$ is of standard Levi type in $\mf{H}$, and hence in $\g_n$ and $\sg_n$.

\begin{lem}
    The algebra $U(\gl(n))$ is a subalgebra of both $U(\g_n,e)$ and $U(\sg_n,e)$, for $n\le\infty$.
\end{lem}

\begin{proof}
We prove for $U(\g_n,e)$ only; the proof for  $U(\sg_n,e)$ is the same. To simplify notation, for the remainder of the proof, we let $\g$ stand for $\g_n$.

In the case when the Dynkin grading is even, we have $\g(0)^e\subseteq U(\g,e)$, which follows from the fact that $\chi([X,Y])=0$, for $X\in\g(0)^e$ and $Y\in\mf m$.
Since the $[e,\gl(n)]=0$ and $\gl(n)\subseteq\g(0)$, the lemma follows in this case.

In the case when the Dynkin grading is not even, in light of the previous discussion, it therefore suffices to prove that, in the construction of the $W$-(super)algebra, we may choose the Lagrangian subspace in $l\subseteq \g(-1)$ so that it is $\mf \gl(n)$-invariant. This in indeed can be accomplished, and can be seen as follows. From the classification of good gradings (which includes all the Dynkin gradings)  of these Lie (super)algebras in \cite{EK05, BG07, Ho12} one sees that $\g(-1)$, as a $\gl(n)$-module, is a direct sum of copies of $\C^n$, $\C^{n*}$ or the trivial representation. Now, the action of $\mf t$ on $\g(-1)$ preserves the symplectic form $\omega_\chi$, and hence $\g(-1)^*\cong\g(-1)$ as a $\mf t$-module. As $\mf t$ contains a Cartan of $\gl(n)$, it follows that $\g(-1)^*\cong\g(-1)$ as a $\gl(n)$-module. As the space $\g(-1)$ is even dimensional, this allows us to choose the subspace $l$ to be $\gl(n)$-invariant.
\end{proof}

Let $e$ be of standard Levi type in $\mf{H}$ as above and suppose that it gives rise to the nilcharacter $\zeta:\mf n\rightarrow\C$. That is, $e$ is principal nilpotent in a Levi subalgebra of $\mf H$, and hence in a Levi subalgebra $\fl$ of $\g_n$ and $\sg_n$, and $\zeta$ is defined to take value $1$ on the simple root vectors in $\mf l$ and $0$ elsewhere, see \eqref{eq:zeta}.
Recall our equivalence of categories for Whittaker categories of $\g_n$-modules $\mc W(\zeta)_n^+$ and of $\sg_n$-modules $\ov{\mc W}(\zeta)_n^+$ in Theorem \ref{thm:SD:MMS}. In the case of finite $n$, they are subcategories of $\mc W(\zeta)_n$ and $\ov{\mc W}(\zeta)_n$, which, in turn, are subcategories of MMS categories $\MS(\zeta)\subseteq\mf {W}(\theta,e)$, respectively. The proof of the super analogue of Losev's decomposition theorem in \cite[Theorem~ 1.6]{SX20} works in the setting of these Lie superalgebras $\g_n$ and $\sg_n$ in the limit $n\to\infty$, as the (super)symplectic spaces $V=[f,\g_n]$ and $V=[f,\sg_n]$ are both finite-dimensional. Thus, we can restrict the equivalence of categories $\mc K$ in \eqref{equiv:mcK} to the subcategories $\mc W(\zeta)_n^+$ and $\ov{\mc W}(\zeta)_n^+$. Therefore, Theorem \ref{thm:SD:MMS} and Theorem~\ref{thm:isoK} imply the following equivalence of categories in the limiting case $n\to\infty$.

\begin{thm}\label{thm:equiv:walg}
    There exists an equivalence of categories between the category $\mc K(\mc W(\zeta)^+)$ of $U(\g,e)$-modules and the category $\mc K(\ov{\mc W}(\zeta)^+)$ of $U(\sg,e)$-modules under which the parabolic Verma modules correspond.
\end{thm}

Below, we shall provide additional detail on the correspondence between the parabolic Verma modules in Theorem \ref{thm:equiv:walg}. For the sake of concreteness, we take an integral element $\theta$ satisfying the following: (1) $\theta$ is non-negative on all simple roots, (2) $\theta$ is regular on $\mf{H}$ and singular on $\ufk$, and (3) $\theta$ is positive on all positive roots not lying in $\mf{H}$ and $\ufk$. Recall the categories $\mc O'(\theta,e)$ and $\ov{\mc O}'(\theta,e)$ of $U(\g,e)$- and $U(\sg,e)$-modules. For $M\in {\mc O}'(\theta,e)$ and $\ov{M}\in \ov{\mc O}'(\theta,e)$, $\text{Res}^{U(\g,e)}_{U(\ufk)}M$ and $\text{Res}^{U(\sg,e)}_{U(\ufk)}\ov{M}$ are direct sums of finite-dimensional irreducible $\ufk$-modules. We define the full subcategories $\mc O(\theta,e)^+\subset \mc O'(\theta,e)$ and $\ov{\mc O}(\theta,e)^+\subset\ov{\mc O}'(\theta,e)$ consisting of objects $M$ such that the $\ufk$-action is polynomial. We then have $\mc K(\mc W(\zeta)^+)\subset \mc O(\theta,e)^+$ and $\mc K(\ov{\mc W}(\zeta)^+)\subset \ov{\mc O}(\theta,e)^+$. As we have observed in the proof of Corollary \ref{cor::verma}, we have $\mc K\circ M^{\theta,e}\cong\mc M^{\theta,e}$, and hence we conclude that the parabolic standard Whittaker modules in $\mc W(\zeta)^+$ (respectively, $\ov{\mc W}(\zeta)^+$) correspond to the parabolic Verma modules in $\mc O'(\theta,e)$ (respectively, $\ov{\mc O}(\theta,e)$) such that the $\ufk$-action is polynomial. They are parabolically induced as in \eqref{eq::ind:w}, where $V$ therein is now replaced by a tensor product of an irreducible polynomial $\ufk$-module and a one-dimensional $U(\fl\cap\mf{H},e)$-module (corresponding, under $\texttt{Sk}$, to an irreducible Kostant $(\fl\cap\mf H)$-module). Since the parabolic standard Whittaker modules in $\mc W(\zeta)^+$ and $\ov{\mc W}(\zeta)^+$ correspond under super duality, it follows that the parabolic Verma modules correspond as well.

\begin{rem}
    Via the equivalences of categories $\mc K$ in \eqref{equiv:mcK}, we can reformulate \eqref{eq:VWzeta} as
    \begin{align*}
[\mc K({\mc W}(\zeta)^+)] &\cong S^{m_\zeta}\VV  \otimes \wedge^{\infty}\VV,
\qquad
[\mc K(\ov{\mc W}(\zeta)^+)] \cong S^{m_\zeta}\VV \otimes \wedge^{\infty}\WW.
    \end{align*}
\end{rem}

\begin{rem}
    \label{rem:GG2}
    For integral central characters, the Grothendieck groups of $\mc K(\mc W(\zeta)^+)$ and $\mc O(\theta,e)^+$  coincide as these two categories have the same simple objects. A parallel statement holds for $\mc K(\ov{\mc W}(\zeta)^+)\subset \ov{\mc O}(\theta,e)^+$.
\end{rem}

\bibliographystyle{alpha}
\bibliography{references.bib}

\end{document}